\documentclass[10pt]{amsart}

\usepackage[nice]{nicefrac}
\usepackage{amsmath}
\usepackage{amsthm}
\usepackage{amssymb}
\usepackage{delarray}
\usepackage[all,cmtip]{xy}
\usepackage[dvips]{graphics}
\usepackage{epsfig}
\usepackage{mathrsfs}
\usepackage{color}
\usepackage{setspace}  
\usepackage{mathptmx}
\usepackage{stmaryrd}
\usepackage{graphicx}
\usepackage{color}
\usepackage{fancybox}
\usepackage{float}

\numberwithin{equation}{section}

\renewcommand\bar\overline

\newtheorem{thm}{Theorem}[section]
\newtheorem{lemma}[thm]{Lemma}

\newtheorem{conj}[thm]{Conjecture}

\theoremstyle{definition}

\theoremstyle{remark}
\newtheorem{rmk}[thm]{Remark}

\begin{document}
\title{Some experiments with integral Apollonian circle packings}
\author{Elena Fuchs}
\address{Princeton University, Department of Mathematics, Fine Hall, Washington Rd, Princeton, NJ 08544-100}
\email{efuchs@math.princeton.edu}
\author{Katherine Sanden}
\address{Princeton University, Department of Mathematics, Fine Hall, Washington Rd, Princeton, NJ 08544-100}
\email{ksanden@math.princeton.edu}
\thanks{K. Sanden was supported by NSF Grant 0758299}

\begin{abstract}
Bounded Apollonian circle packings (ACP's) are constructed by repeatedly inscribing circles into the triangular interstices of a configuration of four mutually tangent circles, one of which is internally tangent to the other three.  If the original four circles have integer curvature, all of the circles in the packing will have integer curvature as well.  In \cite{ll}, Sarnak proves that there are infinitely many circles of prime curvature and infinitely many pairs of tangent circles of prime curvature in a primitive\footnote[1]{A primitive integral ACP is one in which no integer $>1$ divides the curvatures of all of the circles in the packing.} integral ACP.  In this paper, we give a heuristic backed up by numerical data for the number of circles of prime curvature less than $x$, and the number of  ``kissing primes," or {\it pairs} of circles of prime curvature less than $x$ in a primitive integral ACP.  We also provide experimental evidence towards a local to global principle for the curvatures in a primitive integral ACPs.
\end{abstract}

\maketitle
%
%

\section{Introduction}\label{intro}

Start with four mutually tangent circles, one of them internally tangent to the other three as in Fig.~\ref{circles}.  One can inscribe into each of the curvilinear triangles in this picture a unique circle (the uniqueness follows from an old theorem of Apollonius of Perga circa 200 BC).  If one continues inscribing the circles in this way the resulting picture is called an Apollonian circle packing (ACP).  A key aspect of studying such packings is to consider the radii of the circles which come up in a given ACP.  However, since these radii become small very quickly, it is more convenient to study the {\it curvatures} of the circles, or the reciprocals of the radii.   Studied in this way, ACP's possess a beautiful number theoretic property that all of the circles in an ACP have integer curvature if the initial four have integer curvature.  The number theory associated with these integral ACP's  has been investigated extensively in \cite{Apollo}, \cite{Fuchs}, and \cite{oh}.
\begin{figure}[H]
\centering
\includegraphics[height = 30 mm]{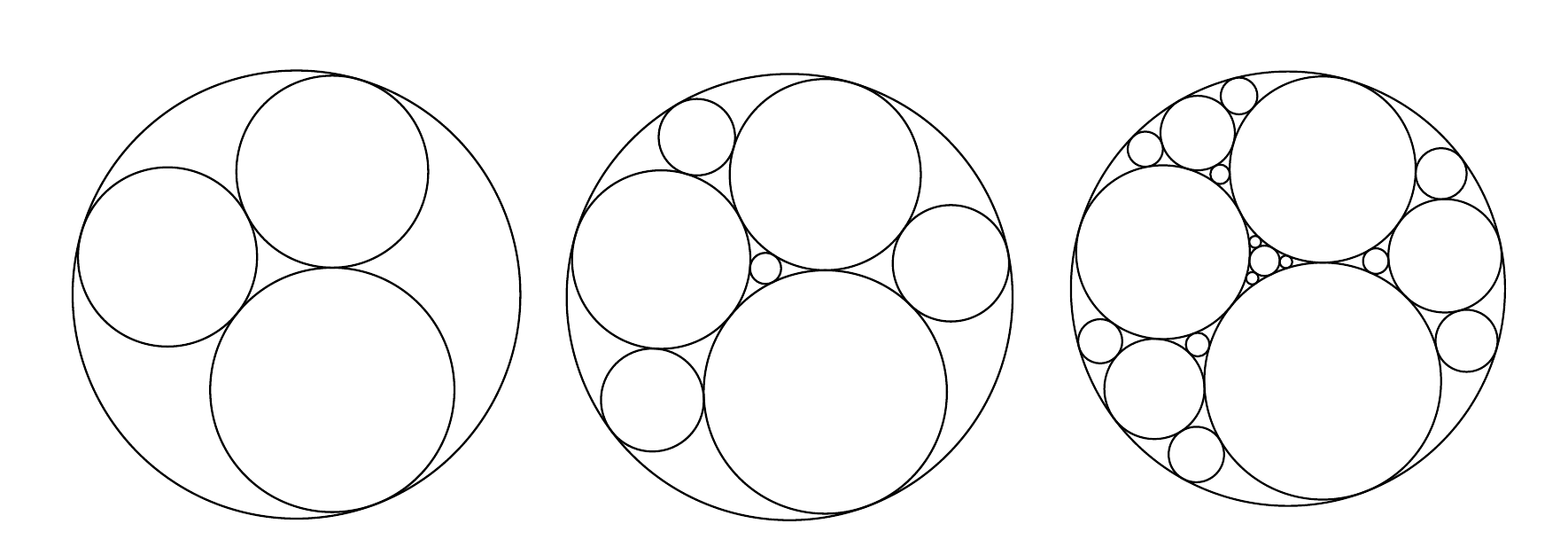}
\caption{Packing Circles}\label{circles}
\end{figure}

A central theorem to any of the results in these papers is Descartes' theorem, which says that the curvatures $(v_1,v_2,v_3,v_4)$ of any four mutually tangent circles satisfy what is called the Descartes equation,
\begin{equation}\label{descartes}
F(v_1,v_2,v_3,v_4)= 2(v_1^2+v_2^2+v_3^2+v_4^2)-(v_1+v_2+v_3+v_4)^2 =0.
\end{equation}
where a circle which is internally tangent to the other three is defined to have negative curvature (see \cite{Coxeter} for a proof).  Given this formula we may assign to every set of $4$ mutually tangent circles in an integral packing $P$ a vector $\mathbf v\in \mathbb Z^4$ of the circles' curvatures.  We use Descartes' equation to express any integral ACP as an orbit of a subgroup of the orthogonal group $\textrm{O}_F(\mathbb Z)$ acting on $\mathbf v$.  This subgroup, called the Apollonian group, is specified in \cite{Apollo}, and we denote it by $A$.  It is a group on the four generators
 \begin{equation}\label{gens}
 \small{
S_1=\left(
\begin{array}{llll}
-1&2&2&2\\
0&1&0&0\\
0&0&1&0\\
0&0&0&1\\
\end{array}
\right)\quad
S_2=\left(
\begin{array}{llll}
1&0&0&0\\
2&-1&2&2\\
0&0&1&0\\
0&0&0&1\\
\end{array}
\right)}
\end{equation}
\begin{equation*}
 \small{
S_3=\left(
\begin{array}{cccc}
1&0&0&0\\
0&1&0&0\\
2&2&-1&2\\
0&0&0&1\\
\end{array}
\right)\quad
S_4=\left(
\begin{array}{cccc}
1&0&0&0\\
0&1&0&0\\
0&0&1&0\\
2&2&2&-1\\
\end{array}
\right),}
\end{equation*}
derived by fixing three of the coordinates of $\mathbf v$ and solving $F(\mathbf v)=0$ for the fourth.  Note that each $S_i$ is of order $2$ and determinant $-1$.  Also, $S_i$ fixes all but the $i$th coordinate of $\mathbf v\in \mathbb Z^4$, producing a new curvature in the $i$th coordinate.

In their paper \cite{Apollo} the five authors Graham, Lagarias, Mallows, Wilks, and Yan ask several fundamental questions about the curvatures in a given integer ACP, which have mostly been resolved in \cite{Fuchs}, \cite{Fuchs1}, \cite{FuchsB}, and \cite{oh}.  In particular, they make some observations about the congruence classes of curvatures which occur in any given ACP, and suggest a ``strong density" conjecture, that every sufficiently large integer satisfying these congruence conditions should appear as a curvature in the packing.  In \cite{FuchsB}, the first author and Bourgain prove a weaker conjecture of Graham et.al. of this flavor that the integers appearing as curvatures in a given ACP make up a positive fraction of $\mathbb N$.  Proving the ``strong density" conjecture would be significantly more difficult.  In this paper, we use the $p$-adic description of the Apollonian orbit from \cite{Fuchs1} to formulate this conjecture precisely and provide strong experimental evidence in Section~\ref{locglobal} in support of it.  Our conjecture is specified further in the case of two different ACP's in Section~\ref{locglobal}.  It is stated generally here.
\begin{conj} {\bf Local to Global Principle for ACP's: }\label{LG}
Let $P$ be an integral ACP and let $P_{24}$ be the set of residue classes mod $24$ of curvatures in $P$.  Then there exists $X_P\in\mathbb Z$ such that any integer $x>X_P$ whose residue mod $24$ lies in $P_{24}$ is in fact a curvature of a circle in $P$.
\end{conj}
We note that $X$ above depends on the packing $P$ under consideration.  In this paper, we investigate two ACP's which we call the {\it Bugeye} packing $P_B$ and the {\it Coins} packing $P_C$.  These packings are represented by right action of the Apollonian group on $(-1,2,2,3)$ and $(-11,21,24,28)$, respectively (see Fig.~\ref{bugcoin} for a picture).  In the case of $P_B$, our data suggests that $X_{P_B}$ exists and is $\leq 10^6$, as we find no integers $x>10^6$ which violate the above conjecture.  The data for $P_C$,  however, suggests that $X_{P_C}$ exists, but that it is $>10^8$.  Namely, there are integers $x>10^8$ in certain  residue classes in the set $S_{24}$ which do {\it not} appear as curvatures in the packing we consider.  We explain this further in Section~\ref{locglobal}.

\begin{figure}[H]
\centering
\includegraphics[height = 65 mm]{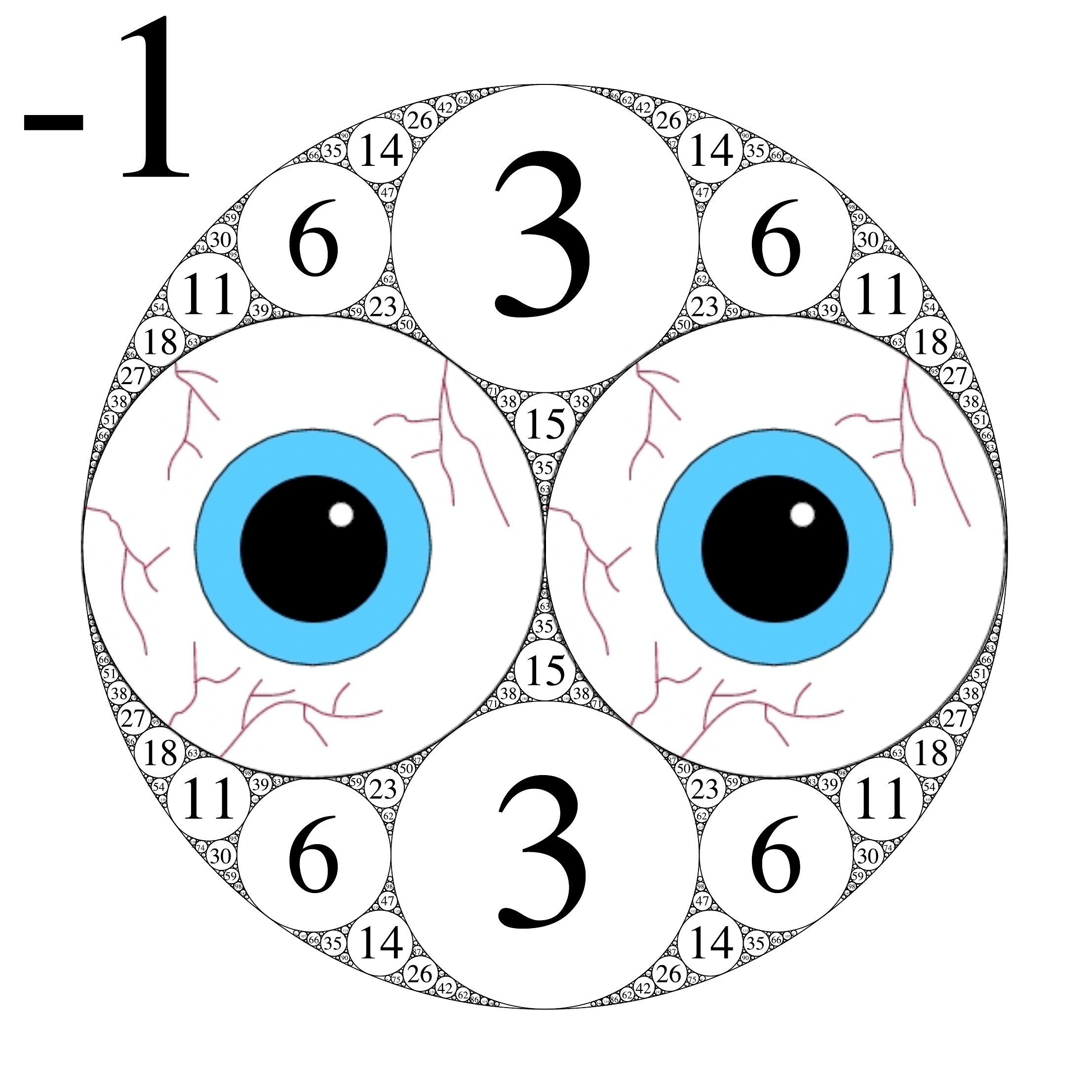}
\qquad\qquad
\includegraphics[height = 65 mm]{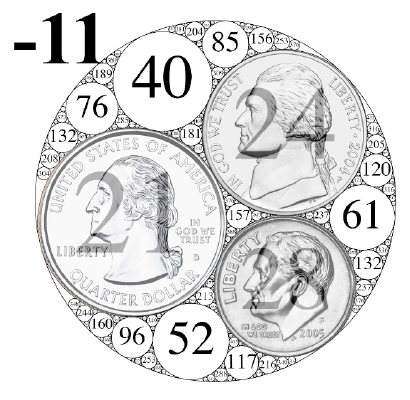}
\caption{{\it Bugeye} and {\it Coins} packings}\label{bugcoin}
\end{figure}

Another interesting problem regarding ACP's is counting circles of prime curvature in a given packing.  Sarnak proves in \cite{ll} that there are infinitely many circles of prime curvature in any packing.  In light of this, we give a heuristic in Section~\ref{primeACP} for the weighted prime count $\psi_P(x)$:
\begin{equation}
\psi_P(x) = \mathop{\sum_{a(C) \leq x}}_{a(C) \;\text{prime}} \log \bigl(a(C)\bigr)
\label{intropsi}
\end{equation}
where $C$ is a circle in the packing $P$ and $a(C)$ is its curvature.  This count is closely related to the number $\pi_P(x)$ of prime curvatures less than $x$ in a packing $P$ (see Remark~\ref{katsremark}).  We confirm experimentally that our heuristic holds for the packings $P_B$ and $P_C$.  We note that our heuristic does not depend on the chosen packing $P$ -- in fact, it yields the correct count of prime curvatures for all of the packings we checked.  We summarize it in the following conjecture:
\begin{conj}\label{pconj}
Let $N_P(x)$ be the number of circles in a packing $P$ of curvature less than $x$, and let $\psi_P(x)$ be as in (\ref{intropsi}).  Then as $x \rightarrow \infty$,
$$\psi_P(x) \sim L(2, \chi_4) \cdot N_P(x)$$
where $L(2, \chi_4)= 0.9159\dots$ is the value of the Dirichlet $L$-series at $2$ with character $\chi_4(p)= 1$ for $p\equiv 1\;(4)$ and $\chi_4(p)=-1$ for $p\equiv 3\;(4)$.\end{conj} 
Sarnak also shows in \cite{maa} that there are infinitely many pairs of tangent circles of prime curvature (we call these {\it kissing primes}).  We address the question of counting kissing primes in $P$ via the weighted sum $\psi_P^{(2)}(x)$:
\begin{equation}\label{introsum2}
\psi_P^{(2)}(x)\quad =\mathop{\sum_{(C,\,C')\in S}}_{a(C),\, a(C')<x}\log (a(C))\cdot \log (a(C'))
\end{equation}
where $S$ is the set of unordered pairs of tangent circles $(C,C')$ of prime curvature in a packing $P$, and $a(C)$ and $a(C')$ denote their respective curvatures.  In this case, it is less obvious what the relation is between $\psi_P^{(2)}(x)$ and the number $\pi_P^2(x)$ of kissing prime circles in a packing $P$ both of whose curvatures are less than $x$.  We therefore stick with $\psi_P^{(2)}(x)$ in our computation:
\begin{conj}\label{tpconj}
Let $\psi_P^{(2)}(x)$ be as in (\ref{introsum2}), and let $N_P(x)$ be the number of circles in a packing $P$ of curvature less than $x$.  Then
$$\psi_P^{(2)}(x)\sim c\cdot {L^2(2,\chi_4)}\cdot N_P(x),$$
where $N_{P}(x)$ is as above and $c= 1.646\dots$ is given by
$$2 \cdot \prod_{p\equiv 3\,(4)}\left(1-\frac{2}{p(p-1)^2}\right).$$
\end{conj}
These heuristics are computed by counting primes in orbits of the Apollonian group, which is possible due to recent results of Bourgain, Gamburd, and Sarnak in \cite{BoGaSa}, as well as the recent asymptotic count of Kontorovich and Oh in \cite{oh} for the number $N_P(x)$.  Our computer experiments were conducted using Java and Matlab, and the programs are available at http://www.math.princeton.edu/\~{}ksanden/ElenaKatCode.html.  A brief description of our algorithm and a discussion of its running time can be found in Section~\ref{algorithm}.
%
%
\subsection{Arithmetic structure of the Apollonian group and its orbit}\label{Prelim}
Since all of the computations and claims in this paper concern the orbit $\mathcal O$ of the Apollonian group $A$ acting on a vector $\mathbf v\in \mathbb Z^4$, we recall the description of the orbit modulo $d$ for any integer $d$ from \cite{Fuchs1}.  We use this description throughout Sections~\ref{primeACP} and \ref{locglobal}.

\begin{thm}(Fuchs): \label{padicorbit}
Let $\mathcal O$ be an orbit of $A$ acting on a root quadruple\footnote[2]{A root quadruple of a packing $P$ is essentially the $4$-tuple of the curvatures of the largest four circles in $P$.  It is well defined and its properties are discussed in \cite{Apollo}.} of a packing, and let $\mathcal O_d$ be the reduction of this orbit modulo an integer $d>1$.  Let $C=\{\mathbf{v} \not = \mathbf{0} \,| F(\mathbf{v})=0\}$ denote the cone without the origin, and let $C_d$ be $C$ over $\mathbb Z/d\mathbb Z$:
$$C_d=\{\mathbf{v}\in\mathbb Z/d\mathbb Z\, | \mathbf{v}\not \equiv \mathbf{0}\, (d),\, F(\mathbf{v})\equiv 0\, (d)\}$$
Write $d=d_1d_2$ with $(d_2,6)=1$ and $d_1 = 2^n3^m$ where $n,m \geq 0$.  Write  $d_1=v_1v_2$ where $v_1=gcd(24,d_1)$. Then 
\begin{itemize}
\item[(i)] The natural projection $\mathcal O_d\longrightarrow \mathcal O_{d_1}\times \mathcal O_{d_2}$ is surjective.
\item[(ii)]Let $\pi:C_{d_1}\rightarrow C_{v_1}$ be the natural projection.  Then $\mathcal O_{d_1} = \pi^{-1}(\mathcal O_{v_1})$.
\item[(iii)]The natural projection $\mathcal O_{d_2} \longrightarrow \prod_{p^r||d_2}\mathcal O_{p^{r}}$ is surjective and $\mathcal O_{p^r}= C_{p^r}$.
\end{itemize}
\end{thm}
\noindent This result is obtained by analyzing the reduction modulo $d$ of the inverse image of the Apollonian group $A$ in the spin double cover of $SO_F$.  We note that Theorem~\ref{padicorbit} implies that the orbit $\mathcal O$ of $A$ has multiplicative structure in reduction modulo $d=\prod_{p^r||d}p^r$ and that it is completely characterized by its reduction mod $24$, or by $\mathcal O_{24}$ in our notation.  This explains the dependence on $P_{24}$ in Conjecture~\ref{LG}.
\vspace{0.3in}

\noindent {\bf Acknowledgements:} We thank Peter Sarnak, Alex Kontorovich, and Kevin Wayne for many insightful conversations and helpful suggestions.

\section{Prime number theorems for ACP's}\label{primeACP}
In \cite{BoGaSa}, Bourgain et.al. construct an affine linear sieve that gives lower and upper bounds for  prime and almost-prime points in the orbits of certain groups.  In this section, we use their analysis to predict precise asymptotics on the number of prime curvatures less than $x$, as well as the number of pairs of tangent circles of prime curvature less than $x$ in a given primitive Apollonian packing $P$.   The conditions associated with the affine linear sieve for $A$ are verified in \cite{BoGaSa}.  We recall the setup below.

Let $a_n=\#\{\mbox{circles of curvature }n {\mbox{ in a bounded packing }}P\}$, and note that $a_n$ is finite since the number of circles of any given radius can be bounded in terms of the area of the outermost circle.  We consider $1\leq n\leq x$ and note that the sum
$$\sum_n a_n = N_P(x),$$
where $N_P(x)$ is the number of circles of curvature less than $x$ and is determined by the asymptotic formula in \cite{oh} (see Lemma~\ref{oh}).    Key to obtaining our asymptotics is computing the averages of progressions mod $d$ of curvatures less than $x$, where $d>1$ ranges over positive square-free integers of suitable size.  To this end, we define
$$X_d= \sum_{n \equiv 0\,(d)}a_n$$
and introduce a multiplicative density function $\beta(d)$ for which 
$$X_d=\beta(d)\cdot N_P(x) +r(A,d)$$
where the remainder $r(A,d)$ is small according to the results in \cite{BoGaSa}.

In the case of ACP's, we define $\beta(d)$ as follows.  Let $\mathcal O$ be an integral orbit of $A$, and let $\mathcal O_d$ be the reduction of $\mathcal O$ modulo $d$ for a square-free positive integer $d$. Then
\begin{equation}\label{betadef}
\beta_j(d)=\dfrac{\#\{\mathbf v \in \mathcal O_d \, | \, v_j = 0\}}
{\#\{\mathbf v \in \mathcal O_d\}}
\end{equation}
where $v_j$ is the $j$th coordinate of $\mathbf v$.  We recall from Theorem~\ref{padicorbit} that the orbit $\mathcal O_d$ has a multiplicative structure which carries over to the function $\beta_j$ so that
$$\beta_j(d)=\prod_{p|d}\beta_j(p).$$
Thus in order to evaluate $\beta_j(d)$ for arbitrary square-free $d$, we need only to determine $\beta_j(p)$ for $p$ prime.  This is summarized in the following theorem.
\begin{lemma}\label{betathm}
Let $d = \prod p_i$ be the prime factorization of a square-free integer $d>1$.  Then
\begin{itemize}
 \item[(i)] $\beta_j(d)=\prod\beta_j(p_i)$ for $1\leq j\leq 4$.
\item[(ii)] For $p\not =2$, we have
$$\beta_j(p) = \beta_k(p) \mbox{ for $1\leq j,k \leq 4.$}$$
\item[(iii)] For any orbit $\mathcal O$ there exist two coordinates, $i$ and $j$, such that
$$\beta_i(2)=\beta_j(2)=1,$$
$$\beta_k(2)=0 {\mbox{ for $k\not= i, j$.}}$$
We say that the $i$th and $j$th coordinates are even throughout the orbit, while the other two coordinates are odd throughout the orbit.
\item[(iv)] For $p\not=2$, let $\beta(p)=\beta_i(p)$ for $1\leq i \leq 4$.  Then
\begin{equation}\label{bvalue}
                \beta(p)= \left\{ \begin{array}{ll}
                			\frac{1}{p+1}& \mbox{for $p\equiv 1$ mod $4$} \\
                			\frac{p+1}{p^2+1} & \mbox{for $p\equiv 3$ mod $4$}\\
                			  \end{array}
                		 \right. \\
\end{equation}
\end{itemize}
\end{lemma}
\begin{proof}
The statements in (i) and (ii) follow from Theorem~\ref{padicorbit}.  Let $\mathbf v$ be the root quadruple (the quadruple of the smallest curvatures) of the packing $P$.  To show (iii), note that any quadruple in a primitive integral ACP consists of two even and two odd curvatures (see \cite{sand} for a discussion). Without loss of generality, assume $\mathbf v = (1,1,0,0)$ mod $2$, so $i=1$ and $j=2$ in this case.  Since the Apollonian group is trivial modulo $2$, we have that every vector in the orbit is of the form
$(1,1,0,0)$
mod $2$, so we have what we want.

To prove (iv), we use results in \cite{Fuchs1} and recall from Theorem~\ref{padicorbit} that $\mathcal O_p$ is the cone $C_p$ for $p>3$.  Thus the numerator of $\beta(p)$ is
$$\#\{\mathbf{v} \in \mathcal O_p \, | \, v_j = 0\}= \#\{(v_1,v_2,v_3)\in\mathbb F_p^3-\{{\bf 0}\}\, | \, F(v_1,v_2,v_3,0)=0\}$$
where $F$ is the Descartes quadratic form and $p>3$.  So the numerator counts the number of nontrivial solutions to the ternary quadratic form obtained by setting one of the $v_i$ in the Descartes form $F(\mathbf v)$ to $0$.  Similarly we have that the denominator of $\beta(p)$ is
$$\#\{\mathbf{v} \in \mathcal O_p\}=\#\{(v_1,v_2,v_3,v_4)\in\mathbb F_p^4-\{{\bf 0}\} \, | \, F(v_1,v_2,v_3,v_4)=0\}$$
where $p>3$.  So the denominator counts the number of nontrivial solutions to the Descartes form.  The number of nontrivial solutions to ternary and quaternary quadratic forms over finite fields is well known (see \cite{Cassels}, for example).  Namely,
\begin{equation}\label{orbit}
\#\{(v_1,v_2,v_3,v_4)\in\mathbb F_p^4-\{{\bf 0}\} \, | \, F(v_1,v_2,v_3,v_4)=0\} =  \left\{ \begin{array}{ll}
                			p^3+p^2-p-1& \mbox{for $p\equiv 1$ mod $4$} \\
                			p^3-p^2+p-1& \mbox{for $p\equiv 3$ mod $4$}\\
                			  \end{array}
                		 \right. \\
\end{equation}
for $p>3$, and
\begin{equation}\label{orbitbeta}
\#\{(v_1,v_2,v_3)\in\mathbb F_p^3-\{{\bf 0}\} \, | \, F(v_1,v_2,v_3,0)=0\}= p^2-1\mbox{ for all odd primes $p$}.
\end{equation}
Combining (\ref{orbit}) and (\ref{orbitbeta}), we obtain the expression in (\ref{bvalue}) for $p>3$.  For $p=3$, we compute $\mathcal O_p$ explicitly and find that there are two possible orbits of $A$ modulo $3$ which are illustrated via finite graphs in Fig.~\ref{mod31} and Fig.~\ref{mod32}.  Both of these orbits consist of $10$ vectors $\mathbf v\in \mathbb Z^4$.  In both orbits, $4$ of the vectors $\mathbf v$ have $v_i=0$ for any $1\leq i\leq 4$.  Thus $\beta(3)=\frac{2}{5}$ as desired.
\end{proof}
\begin{figure}[H]
\centering
\includegraphics[height = 54 mm]{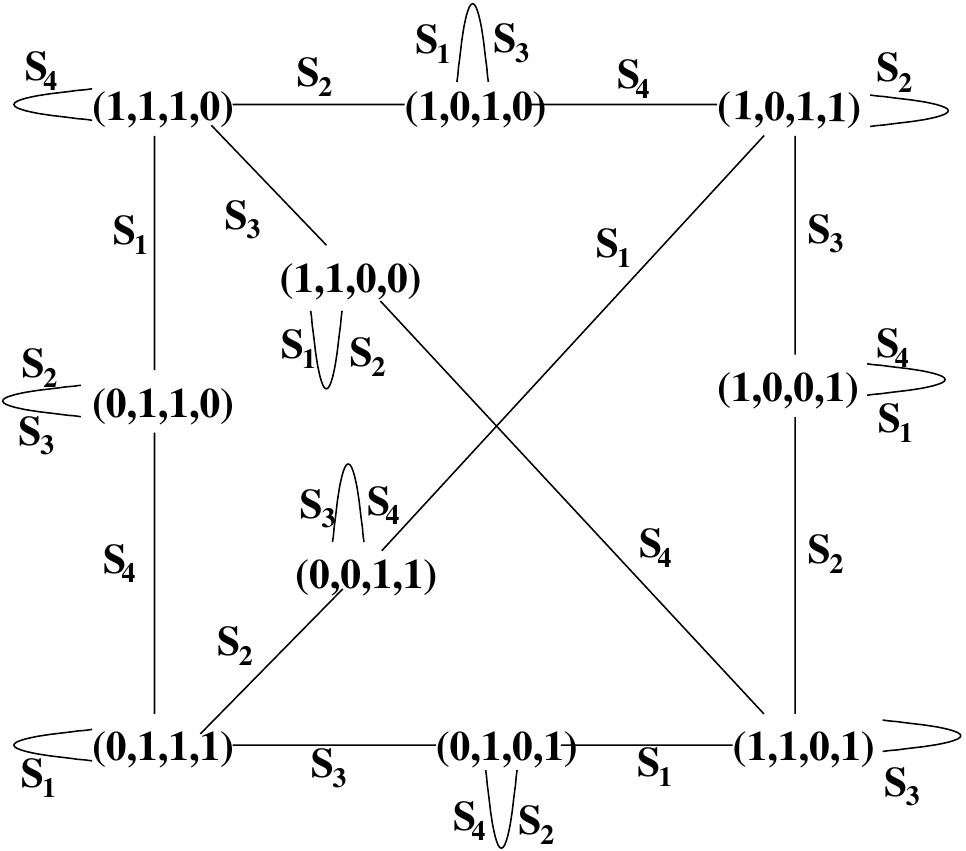}
\caption{Orbit I modulo $3$}\label{mod31}
\end{figure}
\vspace{0.3in}

\begin{figure}[H]
\centering
\includegraphics[height = 54 mm]{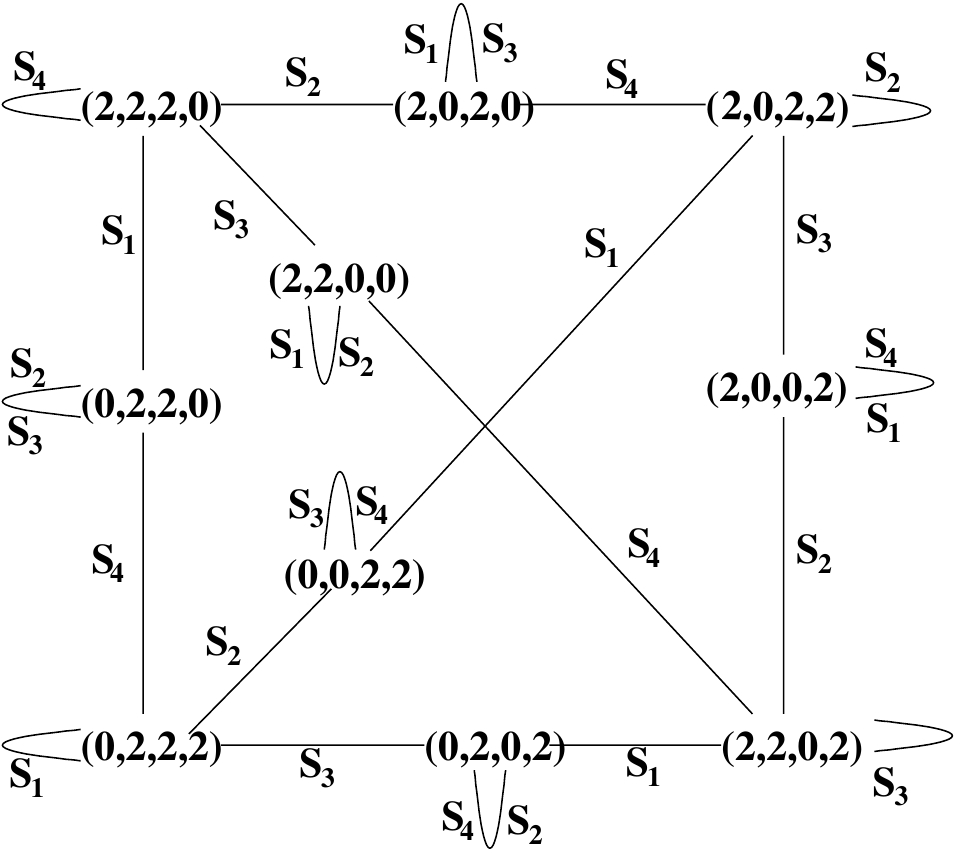}
\caption{Orbit II modulo $3$}\label{mod32}
\end{figure}

In the following two sections, we use this setup to produce a precise heuristic for the number of circles of prime curvature as well as the number of pairs of tangent circles of prime curvature less than $x$ in a given ACP.

%
%
\subsection{Predicting the prime number theorem for ACP's}\label{theory}

In order to compute the number of prime curvatures in an ACP as proposed in Conjecture~\ref{pconj}, we use the setup above paired with properties of the Moebius function to pick out primes in the orbit of $A$ (see (\ref{Lam})).  We use the asymptotic in \cite{oh} for the number $N_P(x)$ of curvatures less than $x$ in a given packing $P$:
\begin{thm}(Kontorovich, Oh): \label{oh}
Given a bounded Apollonian circle packing $P$, there exists a constant $c_P>0$ which depends on the packing, such that as $x\rightarrow\infty$,
$$N_P(x) \sim c_P\cdot x^{\delta},$$
where $\delta = 1.30568\dots$ is the Haussdorf dimension of the limit set of $A$ acting on hyperbolic space.
\end{thm}

For the purpose of our computations, we will need a slightly stronger statement of Theorem~\ref{oh}, as we will sum over each coordinate of the points in the orbit of $A$ separately.  Namely, each circle in the packing is uniquely represented in the orbit $\mathcal O$ as a maximal coordinate of a vector $\mathbf v$ in $\mathbb Z^4$.  We would like to know how many circles there are of curvature less than $x$ that are represented in this way in the $i$th coordinate of a vector in the orbit.  We denote this by $N^{(i)}_P(x)$:
\begin{equation}\label{coordinate}
N^{(i)}_P(x)=\sum_{\stackrel{\mathbf v\in\mathcal O}{v_i^{\ast}\leq x}} 1,
\end{equation}
where $v_i^{\ast}$ denotes the $i$th coordinate of $\mathbf v\in \mathbb Z^4$ which is also a maximal coordinate of $\mathbf v$\footnote[3]{It is possible that there is more than one $i$ for which the $i$th coordinate is maximal}.  To this end we have
\begin{lemma}\label{sullivan}
Let $N^{(i)}_P(x)$ and $N_P(x)$ be as above.  Then
\begin{equation}
N^{(1)}_P(x)\sim N^{(2)}_P(x)\sim N^{(3)}_P(x)\sim N^{(4)}_P(x)\sim \frac{N_P(x)}{4}
\end{equation}
as $x$ approaches infinity.
\end{lemma}
\begin{proof}
The computation in \cite{oh} of the main term in the asymptotics in Theorem~\ref{oh} relies on the Patterson-Sullivan measure on the limit set of the Apollonian group $A$.  In order to prove Lemma~\ref{sullivan}, we show that this measure is invariant under transformations on the coordinates of a vector $\mathbf v$ in an orbit $\mathcal O$ of $A$.

To this end, let $G$ be the group of permutations of the coordinates $v_1, \dots, v_4$ of a vector $\mathbf v\in \mathcal O$.  The group $G$ is finite and its elements can be realized as $4$ by $4$ integer matrices.  For example, the matrix
\begin{equation*}
M=\left(
\begin{array}{cccc}
0&1&0&0\\
1&0&0&0\\
0&0&1&0\\
0&0&0&1\\
\end{array}
\right)
\end{equation*}
switches the first and second coordinates by left action on $\mathbf v^{\textrm{T}}$.  Let
$$L=(G,A)$$
be the group of $4$ by $4$ matrices generated by the Apollonian group $A$ together with $G$, and note that each element of $G$ normalizes $A$.  For example, if $M$ is as above, we have $M^{-1}S_1M = S_2$, where $S_1$ and $S_2$ are as in (\ref{gens}).  Similarly, any element of $G$ switching the $i$th and $j$th coordinate of $\mathbf v$ conjugates $S_i$ to $S_j$ in this way.  Thus $L/A$ is finite, and so $L$ is a finite extension of $A$.  In particular, this implies that the Patterson-Sullivan measure for $L$ is the same as for $A$.  Since $L$ is precisely an extension of $A$ by the permutations of the coordinates of $\mathbf v$, we have that the Patterson-Sullivan measure is invariant under $G$.  Together with Theorem~\ref{oh} and its proof in \cite{oh}, this proves the lemma. 
\end{proof}
Since we are interested in counting the points in $\mathcal O$ for which $v_i^{\ast}$ is prime, we sum over points for which $v_i^{\ast}$ is $0$ modulo some square-free $d$.  It is convenient to count primes in the orbit of $A$ with a logarithmic weight.  To this end, we consider the function
\[
                \Lambda(n)= \left\{ \begin{array}{ll}
                			\log p& \mbox{if $n = p^l$} \\
                			0 & \mbox{otherwise} \\
                			  \end{array}
                		 \right. \\
                \]
for which it is well known that
\begin{equation}\label{Lam}
\Lambda(n) = -\sum_{d|n}{\mu (d) \log d},
\end{equation}
where $\mu(d)$ is the Moebius function.  Using this we write down a concrete expression for the number of prime curvatures less than $x$ in a packing $P$ counted with a logarithmic weight:
\begin{lemma}\label{primtm}
Let $v_i^{\ast}$ be the $i$th coordinate of a vector $\mathbf v$ in $\mathcal O$ such that $v_i^{\ast}$ is the maximal coordinate of $\mathbf v$, and let $\psi_P(x)$ be as before.  Then
\begin{equation}\label{lambdasum1}
\psi_P(x)=-\sum_{1\leq i\leq 4}\sum_{\stackrel{\mathbf v\in \mathcal O}{v_i^{\ast}\leq x}}\Lambda(v_i^{\ast}) + \textrm{O}(x).
\end{equation}
\end{lemma}
The sum in (\ref{lambdasum1}) is a count of all circles whose curvatures are {\it powers} of primes.  Including powers of primes in our count will not affect the final answer significantly.  Namely, let $N_P^{\Box}(x)$ be the number of circles in a packing $P$ whose curvatures are less than $x$ and perfect squares. Note that
$$N_P^{\Box}(x)=\textrm{O}(x).$$
This is insignificant compared to the count of all curvatures in Theorem~\ref{oh}, so the sum in (\ref{lambdasum1}) is the correct one to consider.  Denote by $D<N_P(x)$ be the level distribution from the analysis in \cite{BoGaSa} -- i.e. our moduli $d$ are taken to be less than $D$.  We combine Lemma~\ref{primtm}, Lemma~\ref{sullivan}, and the expression for $\Lambda(n)$ in (\ref{Lam}) to get 
\begin{eqnarray}\label{sumprimes1}
\psi_P(x)&=&-\sum_{1\leq i\leq 4}\sum_{\stackrel{\mathbf v\in \mathcal O}{v_i^{\ast}\leq x}}\sum_{d|v_1^{\ast}}\mu(d)\log d+\textrm{O}(x)\nonumber \\ &=&
-\sum_{1\leq i\leq 4} \sum_{\stackrel{\mathbf v\in\mathcal O}{v_i^{\ast}\leq x}}\sum_{d\leq D}\mu(d)\log d\hspace{-0.1in}\sum_{v_i^{\ast}\equiv 0\;(d)}\hspace{-0.1in}1 - \sum_{1\leq i\leq 4} \sum_{\stackrel{\mathbf v\in\mathcal O}{v_i^{\ast}\leq x}}\sum_{d> D}\mu(d)\log d\hspace{-0.1in}\sum_{v_i^{\ast}\equiv 0\;(d)}\hspace{-0.1in}1 \, +\textrm{O}(x)
\end{eqnarray}
Assuming that the Moebius function $\mu(d)$ above becomes random as $d$ grows, the sum over $d>D$ in (\ref{sumprimes1}) is negligible, and we omit it below.  We proceed by rewriting the sum over $d\leq D$ in (\ref{sumprimes1}) using the density function $\beta(d)$ in (\ref{betadef}).  Recall that the analysis in \cite{BoGaSa} and \cite{oh} gives us that
$$ \sum_{n \equiv 0\,(d)}a_n = \beta(d)\cdot c_P x^{\delta} + r(A,d)$$
where $r(A,d)$ is small on average.  In particular,
$$\sum_{d\leq D}r(A,d) = \textrm{O}(x^{\delta-\epsilon_0})$$
for some $\epsilon_0>0$.  Paired with the assumption that $\mu$ is random, this evaluation of the remainder term allows us to rewrite (\ref{sumprimes1}) as follows:
\begin{eqnarray}\label{sumprimes}
&&{-\sum_{1\leq i\leq 4} \frac{N_P^{i}(x)}{4}}\sum_{d\leq D}\beta_i(d)\mu(d)\log d+\textrm{O}(x^{\delta-\epsilon_0})\nonumber \\ &=&
-\frac{N_P(x)}{4}\sum_{1\leq i\leq 4}\sum_{d\leq D}\beta_i(d)\mu(d)\log d +\textrm{O}(x^{\delta-\epsilon_0})\\\nonumber
\end{eqnarray}
To compute the innermost sum in the final expression above, note that 
\begin{equation}\label{infinity}
\sum_{d\leq D}\beta_i(d)\mu(d)\log d= \sum_{d>0}\beta_i(d)\mu(d)\log d- \sum_{d> D}\beta_i(d)\mu(d)\log d.
\end{equation}
Assuming once again that the sum over $d>D$ is insignificant due to the conjectured randomness of the Moebius function, we have that the sum over $d\leq D$ in (\ref{infinity}) can be approximated by the sum over all $d$.  With this in mind, the following lemma yields the heuristic in Conjecture~\ref{pconj}.
\begin{lemma}\label{infinitesumcalc}
Let $\beta_i(d)$ be as before.  We have
$$\sum_{1\leq i\leq 4}\sum_{d>0}\beta_i(d)\mu(d)\log d = 4\cdot L(2,\chi_4)$$
where $L(2,\chi_4)= 0.91597\dots$ is the value of the Dirichlet $L$-function at $2$ with character
\begin{eqnarray*}
\chi_4(p) = 
\begin{cases}
1 & \text{if $p \equiv 1$ (mod 4)} \\
-1 & \text{if $p \equiv 3$ (mod 4)}
\end{cases}
\end{eqnarray*}
\end{lemma}
\begin{proof}
We introduce a function 
\begin{equation*}
f(s) = \sum_d \beta_i(d) \mu(d) d^{-s},
\end{equation*}
and note that its derivative at $0$ is precisely what we want:
\begin{equation*}\label{Fprimezero}
f'(0) = -\sum_d \beta_i(d) \mu(d) \log d.
\end{equation*}
Since the functions $\beta$, $\mu$, and $d^s$ are all multiplicative, we may rewrite $f(s)$ as an Euler product and obtain
\begin{eqnarray*}
f(s) &=& \prod_p \left( 1 - \beta_i(p) p^{-s} \right)\\&=&
\prod_p (1 - p^{-s-1})\cdot \frac{1 - \beta_i(p) p^{-s}}{1 - p^{-s-1}}\\&=&
\zeta^{-1}(s+1) \cdot\prod_p \frac{1-\beta_i(p)p^{-s}}{1-p^{-s-1}}\\ &=&
\zeta^{-1}(s+1)\cdot H(s), \\
\end{eqnarray*}
where $H(s) = \prod_p  (1-\beta_i(p)p^{-s})({1-p^{-s-1}})^{-1}$ is holomorphic in $\Re (s) > 1/2$.
Differentiating, we obtain
$$f'(0) = -\zeta'(1)\zeta^{-2}(1) \cdot H(0) +\zeta^{-1}(1)\cdot  H'(0) = H(0)$$
since $-\zeta'(1)\zeta^{-2}(1) = 1$ and $\zeta^{-1}(1)=0$.  Thus it remains to compute
$$H(0) = \prod_p \frac{1-\beta_i(p)}{1-p^{-1}}.$$
Part (iii) of Lemma~\ref{betathm} says $\beta_i(2)=\beta_j(2) = 1$ for two coordinates $1\leq i,j\leq 4$.  For these two coordinates $1 - \beta_i(2) = 0$ and so $H(0) = 0.$  Otherwise $\beta_i(2) = 0$ and we have
\begin{eqnarray*}
H(0) &=& \frac{1}{1-\frac{1}{2}}\cdot \prod_{p \equiv 1\, (4)} \left( 1 - \frac{1}{p+1} \right) \frac{1}{1 - p^{-1}} \; \prod_{p \equiv 3 \, (4)} \left(1 - \frac{p+1}{p^2+1} \right)\frac{1}{1 - p^{-1}}\\ &=&
2 \cdot \prod_{p \equiv 1 \,(4)} \frac{p^2}{p^2 - 1}  \prod_{p \equiv 3 \,(4)}\frac{p^2}{p^2+1}\\ \\&=&
 2\cdot L(2,\chi_4).
  \label{productanswer}
\end{eqnarray*}
Thus the sum we wish to compute is $4\cdot L(2,\chi_4)$, as desired.
\end{proof}
Lemma~\ref{infinitesumcalc} implies that the contribution of the two of the coordinates that are even throughout the orbit to the sum in (\ref{sumprimes}) is $0$, and the contribution for the other two coordinates is
$$\frac{N_P(x)}{4}\cdot 4\cdot L(2,\chi_4) = N_P(x)\cdot L(2,\chi_4),$$
yielding the predicted result in Conjecture ~\ref{pconj}.
\begin{rmk}\label{katsremark}
It is well known that $\pi_P(x) \sim \frac{\psi_P(x)}{\log x}$ as $x \rightarrow \infty$.  Thus Conjecture ~\ref{pconj} can also be stated in terms of $\pi_P(x)$:
$$  \pi_P(x) \sim \frac{L(2,\chi_4)\cdot N_P(x)}{\log x}.$$
\end{rmk}

\begin{figure}[H]
\centering
\includegraphics[height = 72 mm]{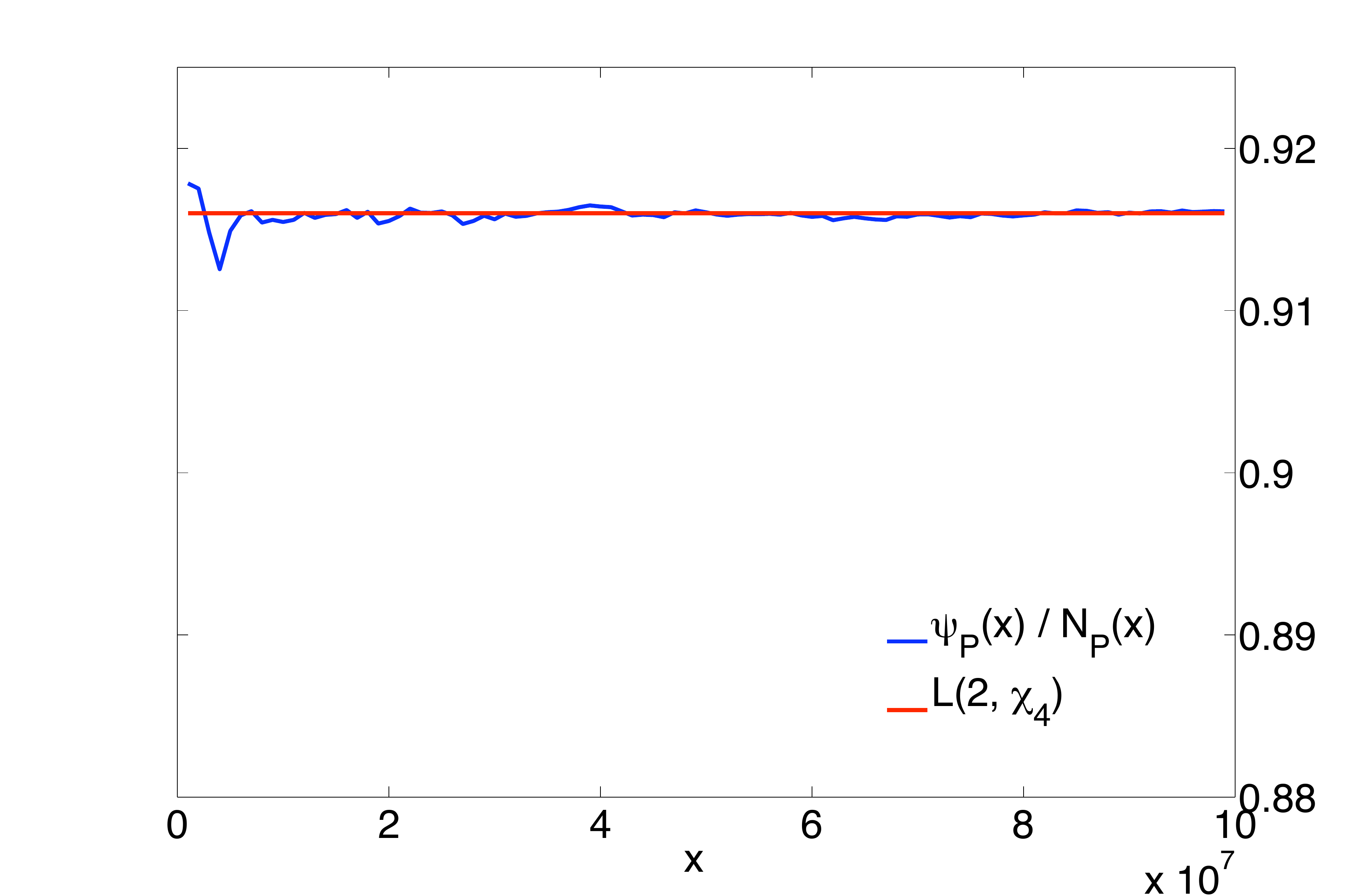}
\caption{Prime Number Heuristic for $P_C$}\label{pntfigure}
\end{figure}

\begin{figure}[H]
\centering
\includegraphics[height = 72 mm]{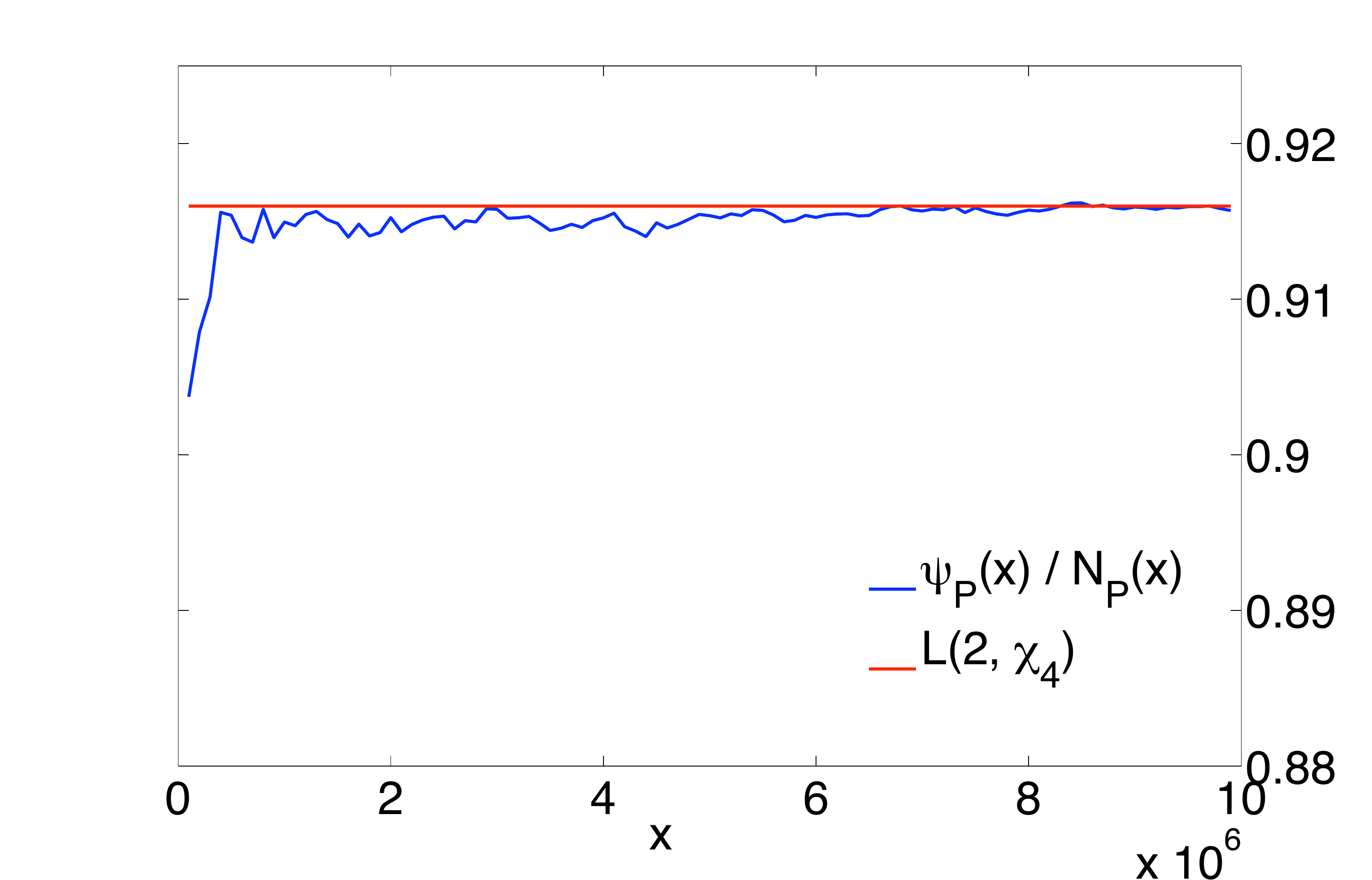}
\caption{Prime Number Heuristic for $P_B$}\label{pntfigureB}
\end{figure}
Since our computations rely on the multiplicativity inherent to the reduction mod $d$ of the Apollonian group $A$, our heuristic is independent of the chosen packing $P$ in which we count prime curvatures.  This is confirmed by our data: Fig.~\ref{pntfigure} and Fig.~\ref{pntfigureB} above show the graphs of
\begin{equation}\label{graphp}
y= \frac{\psi_{P}(x)}{N_{P}(x)}
\end{equation}
where $P=P_C$ and $P_B$ as in Section~\ref{intro} and $x\leq 10^8$.  In both cases, (\ref{graphp}) converges to $y=L(2,\chi_4)$ as predicted in Conjecture~\ref{pconj}.
%
%
%
%
%

\vspace{0.5in}

\subsection{Predicting a prime number theorem for kissing primes}\label{theory2}

In this section, we use the analysis in \cite{BoGaSa}, as well as the conjectured randomness of the Moebius function to arrive at the heuristic in Conjecture~\ref{tpconj} for the number of kissing primes, i.e. pairs of tangent circles both of prime curvature less than $x$.  With the same notation as in Section~\ref{theory}, we would now like to count the points in $\mathcal O$ for which $v_i^{\ast}$ and $v_j$ are prime for some $j\not=i$, so we sum over points for which either $v_i^{\ast}$ or $v_j$ is $0$ modulo some square-free $d$.  To do this we need the total number of pairs of mutually tangent circles of curvature less than $x$ in a packing $P$.  If $N_P(x)$ is the number of circles of curvature up to $x$ as specified by Theorem~\ref{oh}, it is not difficult to see that
\begin{equation}
\#\{{\mbox{pairs of mutually tangent circles of curvature }}a\leq x\}=3\cdot N_P(x)
\end{equation}
since each circle of curvature $a$ in the packing is tangent to a distinct triple of circles of curvature $\leq a$.  This can be seen via a simple proof by induction.  We again employ the function $\Lambda(n)$ in order to write down a concrete expression for the number of kissing primes less than $x$ in a packing $P$.
\begin{lemma}\label{twintm}
Let $v_i^{\ast}$ and $v_j$ be two distinct coordinates of a vector $\mathbf v$ in $\mathcal O$, where $v_i^{\ast}$ denotes the maximum coordinate of $\mathbf v$, and let $\psi_P^{(2)}(x)$ be as before.  Then
\begin{equation}\label{lambdasumm}
\psi_P^{(2)}(x)=\sum_{\stackrel{\mathbf v\in \mathcal O}{v_i^{\ast}\leq x}}\sum_{j\not=i}\Lambda(v_i^{\ast})\Lambda(v_j) + \textrm{O}(x).
\end{equation}
\end{lemma}
Again, the sum in (\ref{lambdasumm}) is a count of all mutually tangent pairs of circles whose curvatures are {\it powers} of primes, but by a similar argument to that in Section~\ref{theory} we have that including powers of primes in our count does not affect the final answer significantly.   Note that in order to evaluate (\ref{lambdasumm}) we introduce in (\ref{referg}) a function which counts points in $\mathcal O$ for which {\it two} of the coordinates are $0$ modulo $p$.  Denote by $D<x$ be the level distribution from the analysis in \cite{BoGaSa} -- i.e. the moduli $d>1$ in the computations below may be taken to be less than $D$.  We rewrite the expression in (\ref{lambdasumm}) using (\ref{Lam}) and get
\begin{eqnarray}\label{sum11}
\qquad&&\sum_{\stackrel{1\leq i,j\leq 4}{i\not=j}}\sum_{\stackrel{\mathbf v\in \mathcal O}{v_i^{\ast}\leq x}}\Biggl(\sum_{d_i|v_i^{\ast}}\mu(d_i)\log d_i\sum_{d_j|v_j}\mu(d_j)\log d_j\Biggr)\\&=&
\sum_{\stackrel{1\leq i,j\leq 4}{i\not=j}}\sum_{\stackrel{\mathbf v\in\mathcal O}{v_i^{\ast}\leq x}}\left(\Sigma^{-}+\Sigma^{+}\right)+\textrm{O}(x)
\end{eqnarray}
where
$$\Sigma^{-} = \Bigl(\sum_{d_i\leq D}\mu(d_i)\log d_i \sum_{v_i^{\ast}\equiv 0\;(d_i)}1\Bigr)\Bigl(\sum_{d_j\leq D}\mu(d_j)\log d_j \sum_{v_j\equiv 0\;(d_j)}1\Bigr),$$
and
$$\Sigma^{+} = \Bigl(\sum_{d_i>D}\mu(d_i)\log d_i \sum_{v_i^{\ast}\equiv 0\;(d_i)}1\Bigr)\Bigl(\sum_{d_j>D}\mu(d_j)\log d_j \sum_{v_j\equiv 0\;(d_j)}1\Bigr).$$
As in Section~\ref{theory}, we omit $\Sigma^{+}$ in (\ref{sum11}) under the assumption that $\mu$ behaves randomly for large values of $d$.  Along with the results about the remainder term in the sieve in \cite{BoGaSa}, the expression in (\ref{sum11}) becomes
\begin{equation}\label{sum1}
{N_P(x)}\sum_{\stackrel{1\leq i,j\leq 4}{i\not=j}}\sum_{[d_i,d_j]\leq D'}\beta_i\left(\frac{d_i}{(d_i,d_j)}\right)\beta_j\left(\frac{d_j}{(d_i,d_j)}\right)g((d_i,d_j))\mu(d_i)\mu(d_j)\log d_i\log d_j+\textrm{O}(x^{\delta-\epsilon_0}),
\end{equation}
where $\beta_i(d)$ is as before, $[d_i,d_j]$ is the least common multiple of $d_i$ and $d_j$, and $(d_i,d_j)$ is their gcd.  The function $g$ above is the ratio
\begin{equation}\label{referg}
g((d_i,d_j)) = \frac{\#\{\mathbf v\in {\mathcal O_{(d_i,d_j)}} \, | \, v_i\equiv 0 \, ((d_i,d_j)) \mbox{ and }v_j\equiv 0 \, ((d_i,d_j))\}}{\#\{\mathbf v\in{\mathcal O_{(d_i,d_j)}}\}}
\end{equation}
where $(d_i,d_j)$ is square-free in our case.  Note that $g(d)$ is multiplicative outside of the primes $2$ and $3$ by Theorem~\ref{padicorbit}, so
$$g(d)=\prod_{p|d}p$$
and we need only to compute $g(p)$ for $p$ prime in evaluating the sum above.
\begin{lemma}\label{gthm}
Let $g(p)$ be as before where $p$ is a prime.  Then
\begin{itemize}
\item[(i)] $g(2)=\left\{ \begin{array}{ll}
                			1& \mbox{if both $v_i$ and $v_j$ are even} \\
                			0& \mbox{if at least one of $v_i$ or $v_j$ is odd} \\
                			  \end{array}
                		 \right.. \\
		 $
\item[(ii)] $
 g(p)= \left\{ \begin{array}{ll}
                			\frac{1}{(p+1)^2}& \mbox{for $p \equiv 1$ mod $4$} \\
                			\frac{1}{p^2+1} & \mbox{for $p\equiv 3$ mod $4$} \\
                			  \end{array}
                		 \right.. \\$
\end{itemize}
\end{lemma}
\begin{proof}
To prove (ii), we note that Theorem~\ref{padicorbit} implies that the numerator of $g(p)$ is
$$\#\{\mathbf v\in {\mathcal O_d} \, | \, v_1\equiv 0 \, (d) \mbox{ and }v_2\equiv 0 \, (d)\} =  \#\{(v_1,v_2)\in\mathbb F_p^2-\{{\bf 0}\}\, | \, F(v_1,v_2,0,0)=0\}$$
for $p>3$.  Thus it is the number of non-trivial solutions to a binary quadratic form with determinant $0$ (the Descartes form in (\ref{descartes}) with two of the $v_i$, $v_j$ set to $0$), and so
$$\#\{\mathbf v\in \mathcal O_p | v_1\equiv 0 \mbox{ mod }p \mbox{ and }v_2\equiv 0 \mbox{ mod }p\} = p-1$$
for all $p>3$ (see \cite{Cassels}, for example).  In the case $p=3$, we observe that in both of the possible orbits of $A$ mod $3$ in Figures~\ref{mod31} and (\ref{mod32}) we have $g(3)=\frac{1}{10}$ as desired.  Part (i) follows from the structure of the orbit $\mathcal O_2$ as observed in the proof of Lemma~\ref{betathm}.
\end{proof}
Denote by $\nu(d_i,d_j)$ the gcd of $d_i$ and $d_j$ (we write just $\nu$ from now on and keep in mind that $\nu$ depends on $d_i$ and $d_j$).  Note that
$$\beta\left(\frac{[d_i,d_j]}{\nu}\right)=\beta_i\left(\frac{d_i}{\nu}\right)\beta_j\left(\frac{d_j}{\nu}\right),$$
where $\frac{d_i}{\nu}$ and $\frac{d_j}{\nu}$ are relatively prime, so we rewrite the expression in (\ref{sum1}) below:
\begin{equation}\label{desired}
\frac{N_P(x)}{4}\sum_{\stackrel{1\leq i,j\leq 4}{i\not=j}}\sum_{[d_i,d_j]\leq D'}\beta_i\left(\frac{d_i}{\nu}\right)\beta_j\left(\frac{d_j}{\nu}\right)g(\nu)\mu(d_i)\mu(d_j)\log d_i\log d_j.
\end{equation}
To compute the sum in (\ref{desired}), we use a similar argument to that in Section~\ref{theory}.  We note that the inner sum in the expression above is equal to
\begin{equation}
\sum_{d_i>0}\sum_{d_j>0} \beta_i\left(\frac{d_i}{\nu}\right)\beta_j\left(\frac{d_j}{\nu}\right)g(\nu)\mu(d_i)\mu(d_j)\log d_i\log d_j - \sum_{[d_i,d_j]> D'}\beta_i\left(\frac{d_i}{\nu}\right)\beta_j\left(\frac{d_j}{\nu}\right)g(\nu)\mu(d_i)\mu(d_j)\log d_i\log d_j
\end{equation}
where we assume the sum over $[d_i,d_j]> D'$ is insignificant by the conjectured randomness of the Moebius function, and thus the sum over all $d_i$ and $d_j$ is a good heuristic for the sum in (\ref{desired}).  We compute the infinite sum in the following lemma and obtain the heuristic in Conjecture~\ref{tpconj}.
\begin{lemma}\label{infinitesum2calc}
Let $\beta_i(d)$, $\beta_j(d)$, and $g(\nu)$ be as before.  We have
$$\sum_{\stackrel{1\leq i,j\leq 4}{i\not=j}}\sum_{d_i>0}\sum_{d_j>0} \beta_i\left(\frac{d_i}{\nu}\right)\beta_j\left(\frac{d_j}{\nu}\right)g(\nu)\mu(d_i)\mu(d_j)\log d_i\log d_j = 8 \cdot L^2(2,\chi_4)\cdot \prod_{p\equiv 3\,(4)}\left(1-\frac{2}{p(p-1)^2}\right).$$
\end{lemma}
\begin{proof}
We introduce the function
\begin{equation*}
f(s_i,s_j) = \sum_{d_i}\sum_{d_j} \beta_i\left(\frac{d_i}{\nu}\right)\beta_j\left(\frac{d_j}{\nu}\right)g(\nu)\mu(d_i)\mu(d_j)d_i^{s_i}d_j^{s_j},
\end{equation*}
and note that
\begin{equation*}
\left.\frac{\partial^2 f(s_i,s_j)}{\partial s_i\partial s_j}\right|_{(0,0)}=\sum_{d_i}\sum_{d_j}\beta\left(\frac{d_i}{\nu}\right)\beta\left(\frac{d_j}{\nu}\right)g(\nu)\mu(d_i)\mu(d_j)\log d_i\log d_j,
\end{equation*}
which is a good heuristic for the sum in (\ref{desired}) as $x$ tends to infinity.  The difficulty in computing this is the interaction of $d_i$ and $d_j$ in $g(\nu)$.  To this end, write $d_i=\nu e_i$ and $d_j=\nu e_j$ where $(e_i,e_j)=1$.  This gives us the following formula for $f(s_i,s_j)$.
\begin{eqnarray}
f(s_i,s_j)&=&
\sum_{\nu}g(\nu)\sum_{(e_i,e_j)=1}\beta_i(e_i)\beta_j(e_j)\mu(\nu e_i)\mu(\nu e_j)(\nu e_i)^{s_i}(\nu e_j)^{s_j}\nonumber\\&=&
\sum_{\nu}g(\nu)\nu^{-(s_i+s_j)}\sum_{e_i, e_j}\sum_{m|(e_i,e_j)} \mu(m)\beta_i(e_i)\beta_j(e_j)\mu(\nu e_i)\mu(\nu e_j) e_i^{s_i}e_j^{s_j}\\
&=& \sum_{\nu}g(\nu)\nu^{-(s_i+s_j)}\sum_m\mu(m)\mu^2(\nu m)m^{-(s_i+s_j)}\nonumber \\
&&\cdot \sum_{\stackrel{(b_i,\nu m)=1}{(b_j,\nu m)=1}}\beta_i(mb_i)\beta_j(mb_j)\mu(b_i)\mu(b_j) b_i^{s_i}b_j^{s_j}\nonumber \\
&=& 
\sum_{\nu}g(\nu)\mu^2(\nu)\nu^{-(s_i+s_j)}\sum_{(m,\nu)=1}\mu(m)\beta_i(m)\beta_j(m)m^{-(s_i+s_j)}\,A(\nu, m, s_i, s_j)\nonumber\\\nonumber
\end{eqnarray}
where
\begin{eqnarray}
A(\nu, m, s_i, s_j)&=&\sum_{\stackrel{(b_i,\nu m)=1}{(b_j,\nu m)=1}}\beta_i(b_i)\beta_j(b_j)\mu(b_i)\mu(b_j)b_i^{-s_i}b_j^{-s_j}  \nonumber \\&=&
\prod_{\stackrel{p_i|\nu m}{p_j|\nu m}}\left((1-\beta_i(p_i)p_i^{-s_i})(1-\beta_j(p_j)p_j^{-s_j})\right)^{-1}\nonumber\\
&&\cdot 
 \prod_{p_i,p_j}(1-\beta_i(p_i)p_i^{-s_i})(1-\beta_j(p_j)p_j^{-s_j}) \\&=&
P_{\nu}(s_i,s_j)\cdot P_m(s_i,s_j)\nonumber \\
&&\cdot \zeta^{-1}(s_i+1)\prod_{p_i}\frac{1-\beta_i(p_i)p_i^{-s_i}}{1-p_i^{-s_i-1}}\zeta^{-1}(s_j+1)\prod_{p_j}\frac{1-\beta_j(p_j)p_j^{-s_j}}{1-p_j^{-s_j-1}}\nonumber \\&=&
P_{\nu}(s_i,s_j)\cdot P_m(s_i,s_j)\cdot \zeta^{-1}(s_i+1)\zeta^{-1}(s_j+1)\cdot B(s_i,s_j),\nonumber\\\nonumber 
\end{eqnarray}
where
$$P_{\nu}(s_i,s_j) = \prod_{\stackrel{p_i|\nu}{p_j|\nu}}\left((1-\beta_i(p_i)p_i^{-s_i})(1-\beta_j(p_j)p_j^{-s_j})\right)^{-1},$$
$$P_m(s_i,s_j) = \prod_{\stackrel{p_i|m}{p_j|m}}\left((1-\beta_i(p_i)p_i^{-s_i})(1-\beta_j(p_j)p_j^{-s_j})\right)^{-1},$$
$$B(s_i,s_j) = \prod_{p_i}\frac{1-\beta_i(p_i)p_i^{-s_i}}{1-p_i^{-s_i-1}}\prod_{p_j}\frac{1-\beta_j(p_j)p_j^{-s_j}}{1-p_j^{-s_j-1}}.$$
We write
$$C(\nu,s_i,s_j) = g(\nu)\nu^{-(s_i+s_j)},$$
$$D(m, s_i,s_j) = \mu(m)\beta_i(m)\beta_j(m)m^{-(s_i+s_j)}$$
which gives us
\begin{eqnarray}
f(s_i,s_j)&=&\sum_{\nu}C(\nu, s_i, s_j)\cdot P_{\nu}(s_i,s_j)\sum_{(m,\nu)=1}D(m,s_i,s_j)\cdot P_m(s_i,s_j)\nonumber \\&&\cdot B(s_i,s_j)\cdot \zeta^{-1}(s_i+1)\zeta^{-1}(s_j+1).
\end{eqnarray}
Note that the sums over $\nu$ and $m$, as well as $B(s_i,s_j)$ converge and are holomorphic.  We now compute the desired derivative. Write
$$G_{i,j}(\nu,m,s_i,s_j)=\sum_{\nu}C(\nu, s_i, s_j)\cdot P_{\nu}(s_i,s_j)\sum_{(m,\nu)=1}D(m,s_i,s_j)\cdot P_m(s_i,s_j).$$
Then we have
\begin{eqnarray}\label{shiit2}
\left.\frac{\partial^2 f(s_i,s_j)}{\partial s_i\partial s_j}\right|_{(0,0)}=\zeta^{-2}(1)\left.\frac{\partial^2 (G_{i,j}(\nu,m,s_i,s_j))}{\partial s_i\partial s_j}\right|_{(0,0)} +\, G_{i,j}(\nu,m,0,0)\cdot B(0,0)\left(\frac{\zeta'(1)}{\zeta^2(1)}\right)^2\\ \nonumber
\end{eqnarray}
Since $\zeta^{-2}(1)=0$, we need not compute the partial derivative of $G$.  If $\beta_i(2)=\beta_j(2)=0$ in the expression for $B$, we have that $B(0,0) = 4\cdot L^2(2,\chi_4)$, since $\frac{\zeta'(1)}{\zeta^2(1)} = 2\cdot L(2,\chi_4)$ as in Section~\ref{theory}.  This holds for the sums in (\ref{desired}) over $(v_i^{\ast}, v_j)$ and $(v_j^{\ast}, v_i)$ where the $i$th and $j$th coordinates in our orbit are everywhere odd.

The contribution to (\ref{desired}) from the terms where $v_i^{\ast}$ or $v_j$ is even is $0$:

\begin{lemma}\label{zero}
Let $\mathcal O$ be an orbit of the Apollonian group.  Given that the $i$-th or $j$-th coordinate of each vector $\mathbf v\in\mathcal O$ is even, we have
$$G_{i,j}(\nu,m,0,0)\cdot B(0,0)=0$$
\end{lemma}

\begin{proof}
Recall that two of the coordinates of the vectors in $\mathcal O$ are always even, and two are odd.  We write
\begin{eqnarray}\label{split}
G_{i,j}(\nu,m,0,0)\cdot B(0,0)&=& B(0,0)\left(\sum_{2|\nu} C(\nu, 0, 0)\cdot P_{\nu}(0,0)\sum_{(m,\nu)=1}D(m,0,0)\cdot P_m(0,0)\right.\nonumber \\&&+ \left. \sum_{(\nu,2)=1} C(\nu, 0, 0)\cdot P_{\nu}(0,0)\sum_{(m,\nu)=1}D(m,0,0)\cdot P_m(0,0)\right).
\end{eqnarray}

\vspace{0.2in}
\noindent{\it Case 1: only one of the $i$-th and $j$-th coordinates is odd throughout the orbit.}

Recall from Lemma~\ref{gthm} that $g(2)=0$ in the case that only one of the coordinates $(i,j)$ is even throughout the orbit, and $g(2)=1$ if both coordinates are even throughout the orbit.  So, if only one of the coordinates $(i,j)$ is even throughout the orbit, we have
$$B(0,0)\cdot\left(\sum_{2|\nu} C(\nu, 0, 0)\cdot P_{\nu}(0,0)\sum_{(m,\nu)=1}D(m,0,0)\cdot P_m(0,0)\right)=0.$$
Recalling that $\beta_i(p)=\beta_j(p)$ for $p>2$ from Lemma~\ref{betathm}, we have
\begin{eqnarray}\label{2m}
&&B(0,0)\left(\sum_{(\nu,2)=1}C(\nu, 0, 0)\cdot P_{\nu}(0,0)\Bigl(\sum_{(m,\nu)=1}D(m,0,0)\cdot P_m(0,0)\Bigr)\right) \\
&=& \prod_{p_i,p_j}\left(\frac{(1-\beta_i(p_i))(1-\beta_j(p_j))}{(1-p_i^{-1})(1-p_j^{-1})}\right)\nonumber\\
&&\cdot\Biggl(\sum_{(\nu,2)=1}C(\nu, 0, 0)\cdot P_{\nu}(0,0)\Bigl(\sum_{\stackrel{(m,\nu)=1}{(m,2)=1}}D(m,0,0)\prod_{p|m}(1-\beta_i(p))^{-2}\Bigr)\Biggr)\nonumber\\
&+& B(0,0)\cdot \left(\sum_{(\nu,2)=1}C(\nu, 0, 0)\cdot P_{\nu}(0,0)\Bigl(\sum_{(2m,\nu)=1}\mu(2m)\beta_i(2m)\beta_j(2m)\cdot P_m(0,0)\Bigr)\right)\nonumber \\
&=& 0 + 0 \nonumber \\ &=& 0\nonumber\\ \nonumber
\end{eqnarray}
since Lemma~\ref{betathm} implies either $1-\beta_i(2)=0$ or $1-\beta_j(2)=0$ in the first term, and either $\beta_i(2m)=0$ or $\beta_j(2m)=0$ in the second term in (\ref{2m}).  We now compute the expression in (\ref{split}) in the case that both the $i$th and $j$th coordinate are even throughout the orbit.

\vspace{0.2in}
\noindent {\it Case 2: the $i$-th and $j$-th coordinates are both even throughout the orbit.}

In this case, Lemma~\ref{gthm} implies that $g(2\nu)=g(\nu)$ and that $C(2\nu,0,0)=C(\nu,0,0)$ for odd $\nu$.  Also, we again have that $\beta_i(p)=\beta_j(p)$ for $p>2$, so the first sum in (\ref{split}) is  
\begin{eqnarray}\label{first}
&&B(0,0)\cdot\left(\sum_{2|\nu} C(\nu, 0, 0)\cdot P_{\nu}(0,0)\Bigl(\sum_{(m,\nu)=1}D(m,0,0)\cdot P_m(0,0)\Bigr)\right)\\
&=&\left(\sum_{(\nu,2)=1} \bigl(C(2\nu, 0, 0)\prod_{p|\nu}(1-\beta_i(p))^{-2}\bigr)\Bigl(\sum_{(m,2\nu)=1}\bigl(D(m,0,0)\prod_{p|m}(1-\beta_i(p))^{-2}\bigr)\Bigr)\right)\nonumber \\
&&\cdot \prod_{p\not=2}\left(\frac{1-\beta_i(p)}{1-p^{-1}}\right)^2\nonumber\\
&=&\left(\sum_{(\nu,2)=1} \bigl(C(\nu, 0, 0)\prod_{p|\nu}(1-\beta_i(p))^{-2}\bigr)\Bigl(\sum_{(m,2\nu)=1}\bigl(D(m,0,0)\prod_{p|m}(1-\beta_i(p))^{-2}\bigr)\Bigr)\right)\cdot L^2(2,\chi_4)\nonumber
\end{eqnarray}
To compute the second sum in (\ref{split}), we note that $\beta_i(2m)=\beta_j(2m)=\beta_i(m)$ by Lemma~\ref{betathm} and write
\begin{eqnarray*}
&&B(0,0)\left(\sum_{(\nu,2)=1}C(\nu, 0, 0)\cdot P_{\nu}(0,0)\Bigl(\sum_{(m,\nu)=1}D(m,0,0)\cdot P_m(0,0)\Bigr)\right)\\
&=& M_1(\nu,0,0)+M_2(\nu,0,0),
\end{eqnarray*}
where 
$$M_1(\nu,0,0) = \Biggl(\sum_{(\nu,2)=1}C(\nu, 0, 0)\cdot P_{\nu}(0,0)\Bigl(\sum_{(m,2\nu)=1}\mu(m)\beta_i^2(m)\cdot\prod_{p|m}(1-\beta_i(p))^{-2}\Bigr)\Biggr) \cdot B(0,0)$$
$$M_2(\nu,0,0) = \left(\sum_{(\nu,2)=1}C(\nu, 0, 0)\cdot P_{\nu}(0,0)\Bigl(\sum_{(m,2\nu)=1}\mu(2m)\beta_i^2(m)\cdot \prod_{p|m}(1-\beta_i(p))^{-2}\Bigr)\right)\cdot B'(0,0)$$
where
$$B'(0,0)=\prod_{p\not=2}\left(\frac{1-\beta_i(p)}{1-p^{-1}}\right)^2$$
Note that, since $1-\beta_i(2)=1-\beta_j(2)=0$, we have $B(0,0)=0$ in this case, so
$$M_1(\nu,0,0) =0.$$
On the other hand,
\begin{eqnarray}\label{second}
&&M_2(\nu,0,0)\\
&=&\left(\sum_{(\nu,2)=1}C(\nu, 0, 0)\cdot P_{\nu}(0,0)\Bigl(\sum_{(m,2\nu)=1}-\mu(m)\beta_i^2(m)\cdot \prod_{p|m}(1-\beta_i(p))^{-2}\Bigr)\right)\cdot B'(0,0)\nonumber\\
&=&-\left(\sum_{(\nu,2)=1} \bigl(C(\nu, 0, 0)\prod_{p|\nu}(1-\beta_i(p))^{-2}\bigr)\Bigl(\sum_{(m,2\nu)=1}\bigl(D(m,0,0)\prod_{p|m}(1-\beta_i(p))^{-2}\bigr)\Bigr)\right)\cdot L^2(2,\chi_4)\nonumber
\end{eqnarray}
Combining (\ref{first}) and (\ref{second}), we have that the expression in (\ref{split}) is $0$, as desired.
\end{proof}
Therefore, we have that in the contributing terms of (\ref{desired}) both $v_i^{\ast}$ and $v_j$ are odd.  This makes up two of the terms in (\ref{desired}), so we have that the sum we wish to compute is equal to
\begin{equation}\label{contterms}
2\cdot G_{i,j}(\nu,m,0,0)\cdot 4 \cdot L^2(2,\chi_4)=
8\cdot L^2(2,\chi_4)\cdot G_{i,j}(\nu,m,0,0).
\end{equation}
 It remains to compute $G_{i,j}(\nu,m,0,0)$ in this case.  Recall that $\beta_i(2)=\beta_j(2)=0$ in this case, and that $\beta_i(p)=\beta_j(p)$ for $p>2$.  We have
\begin{eqnarray}\label{B00}
&&\sum_{(m,\nu)=1}\mu(m)\beta_i^2(m)\prod_{p|m}(1-\beta_i(p))^{-2}\nonumber \\ &=&
\prod_p 1-\frac{\beta_i^2(p)}{(1-\beta_i(p))^2}\prod_{p|\nu}\left(1-\frac{\beta_i^2(p)}{(1-\beta_i(p))^2}\right)^{-1} \\ &=&
\prod_{p\equiv 1\,(4)}1-\frac{1}{p^2}\prod_{p\equiv 3\,(4)}1-\frac{(p+1)^2}{(p^2-p)^2}\prod_{p|\nu}\left(1-\frac{\beta_i^2(p)}{(1-\beta_i(p))^2}\right)^{-1}\nonumber \\\nonumber 
\end{eqnarray}
We write
$$\sigma = \prod_{p\equiv 1\,(4)}1-\frac{1}{p^2}\prod_{p\equiv 3\,(4)}1-\frac{(p+1)^2}{(p^2-p)^2}$$
and get
\begin{eqnarray}\label{A00}
&&G_{i,j}(\nu,m,0,0)\nonumber\\&=&
\sigma\cdot \sum_{\nu}C(\nu,0,0)P_{\nu}(0,0)\cdot\prod_{p|\nu}\left(1-\frac{\beta^2(p)}{(1-\beta(p))^2}\right)^{-1}\nonumber \\ &=&
\sigma\cdot\sum_{\nu}\prod_{p|\nu}\frac{g(p)}{(1-\beta^2(p)(1-\beta(p))^{-2})(1-\beta(p))^2}\\&=&
\sigma\cdot\prod_p 1+\frac{g(p)}{(1-\beta^2(p)(1-\beta(p))^{-2})(1-\beta(p))^2}\nonumber \\ &=&
\sigma\cdot\prod_{p\equiv 1\, (4)}\left(1-\frac{1}{p^2}\right)^{-1}\prod_{p\equiv 3 \, (4)} 1+\frac{p^2+1}{p^4-2p^3-2p-1}\nonumber\\ &=&  \prod_{p\equiv 3\,(4)}\left(1-\frac{(p+1)^2}{(p^2-p)^2}\right)\left(1+\frac{p^2+1}{p^4-2p^3-2p-1}\right)\nonumber
\end{eqnarray}
Therefore our infinite sum is equal to 
\begin{equation}\label{answer}
8\cdot L^2(2,\chi_4)\cdot \prod_{p\equiv 3\,(4)}1-\frac{2}{p(p-1)^2}
\end{equation}
as desired.
\end{proof}
Conjecture~\ref{tpconj} follows from our assumption that the Moebius function is random and from Lemma~\ref{infinitesumcalc2}.  namely, we predict
$$\psi_P^{(2)}(x)\approx 8\cdot\frac{N_P(x)}{4}\cdot L^2(2,\chi_4)\cdot \prod_{p\equiv 3\,(4)}1-\frac{2}{p(p-1)^2} = c\cdot L^2(2,\chi_4),$$
where $c=1.646\dots$ as in Conjecture~\ref{tpconj}.

\begin{figure}[H]
\centering
\includegraphics[height = 62 mm]{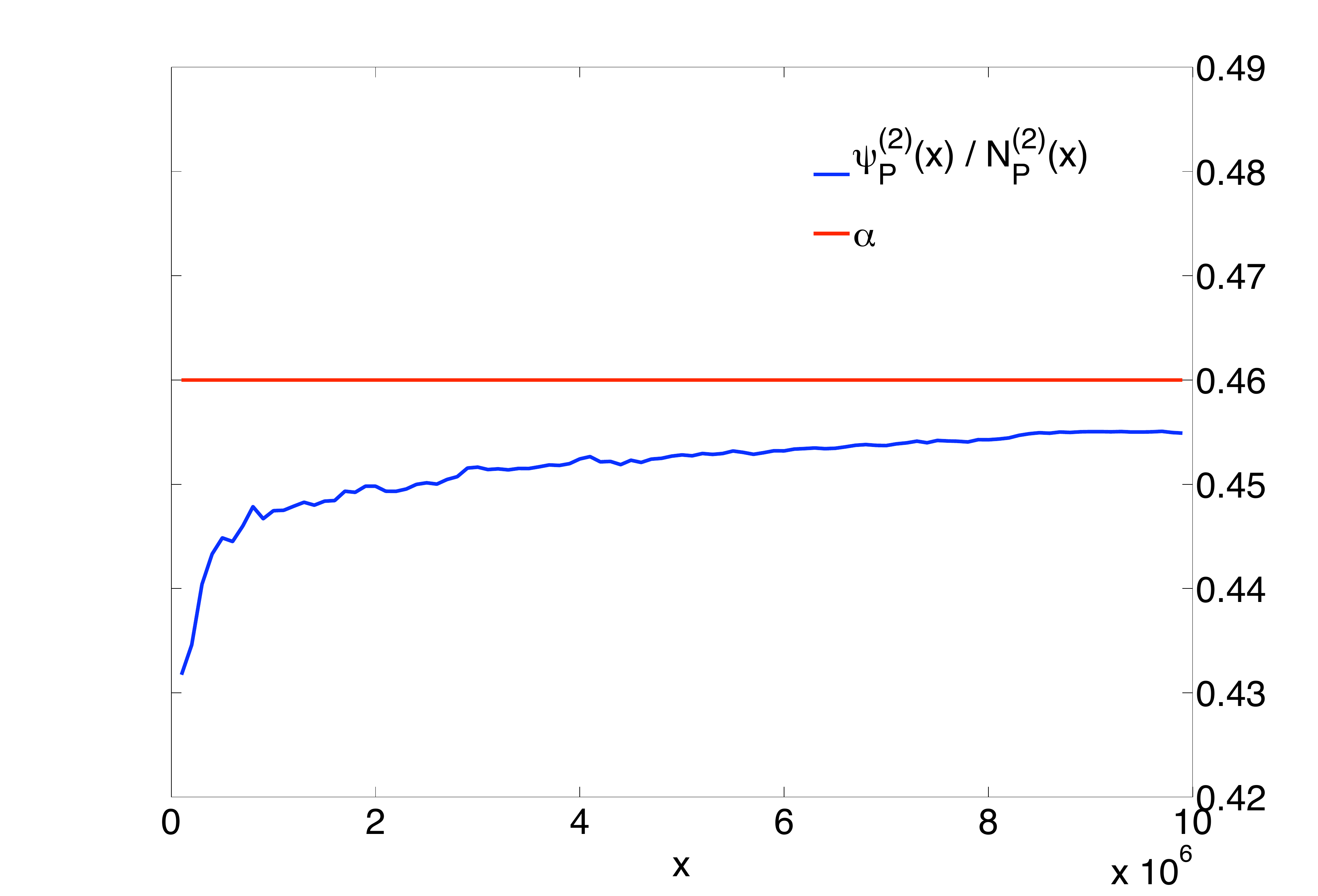}
\caption{Prime Number Theorem for Kissing Primes for the packing $P_B$}\label{kpnt}
\end{figure}

\begin{figure}[H]
\centering
\includegraphics[height = 62 mm]{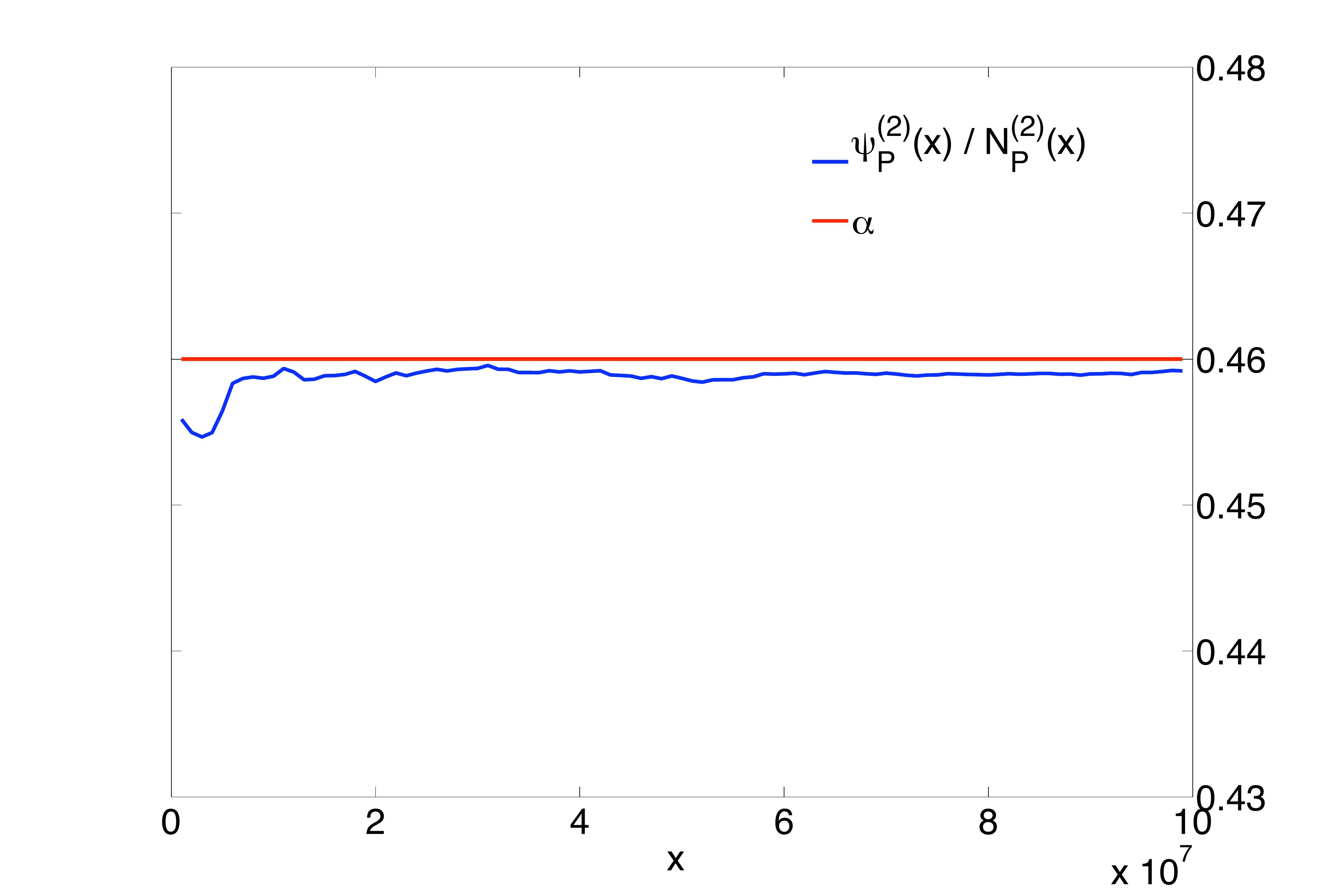}
\caption{Prime Number Theorem for Kissing Primes for the packing $P_C$}\label{kpntC}
\end{figure}

As with our heuristic for the number of prime curvatures less than $x$ in a packing, this count does not depend on the packing $P$.  Fig.~\ref{kpnt} and Fig.~\ref{kpntC} above show graphs of
$$y=\frac{\psi^{(2)}_{P}(x)}{N^{(2)}_{P}(x)}$$
for $P=P_B$ and $P_C$, and $x\leq 10^8$.  Both of these graphs converge to $\alpha = \frac{ c\cdot L^2(2,\chi_4)}{3}$ as conjectured, even though they converge slower than in the case of counting circles of prime curvature in Section~\ref{theory}. 

\section{Local to Global Principle for ACP's}\label{locglobal}

In this section we present numerical evidence in support of Conjecture~\ref{LG}, which predicts a local to global principle for the curvatures in a given integral ACP.  Since the Apollonian group $A$ is small -- it is of infinite index in $O_F(\mathbb Z)$ -- it is perhaps surprising that its orbit should eventually cover all of the integers outside of the local obstruction mod $24$ as specified in Theorem~\ref{padicorbit}. Proving this rigorously, however, appears to be very difficult.  An analogous problem over $\mathbb Z$ would be to show that all large integers satisfying certain local conditions are represented by a general ternary quadratic form -- this analogy is realized by fixing one of the curvatures in Descartes' form and solving the problem for the resulting ternary form.  While this problem has been recently resolved in general in \cite{Co} and \cite{DS}, even there the local to global principle comes in a much more complicated form, relying on congruence obstructions specified in the spin double cover.  Our conjecture, which therefore has the flavor of Hilbert's 11th problem for an indefinite form, predicts a local to global principle of a more straightforward nature.

Our computations suggest that this conjecture is true and we predict the value $X_P$ in the examples we check.  We consider the packings $P_B$ and $P_C$ introduced in Section~\ref{intro}.  Recall that $P_B$ corresponds to the orbit of $A$ acting on $(-1,2,2,3)$, and $P_C$ corresponds to the orbit of $A$ acting on $(-11,21,24,28)$.

In order to explain the data we obtain in both cases, we use Theorem~\ref{padicorbit} to determine the congruence classes (mod $24$) in the given packing.  Recall that the Apollonian group $A$ is generated by the four generators $S_i$ in (\ref{gens}).  We can view an orbit of $A$ modulo $24$ as a finite graph $\mathcal G_{24}$ in which each vertex corresponds to a distinct (mod $24$) quadruple of curvatures, and two vertices $\mathbf v$ and $\mathbf{v'}$ are joined by an edge iff $S_i\mathbf v=\mathbf{v'}$ for some $1\leq i\leq 4$.  Recall from Theorem~\ref{padicorbit} that for any orbit $\mathcal O$ of the Apollonian group,
\begin{equation}\label{2483}
\mathcal O_{24}=\mathcal O_8\times\mathcal O_3,
\end{equation}
so the graph $\mathcal G_{24}$ is completely determined by the structure of $\mathcal O_3$ and $\mathcal O_8$.  There are only two possible orbits modulo $3$, pictured in Fig.~\ref{mod31} and Fig.~\ref{mod32}.  There are many more possible orbits modulo $8$, and we provide the graphs for these orbits in the case of $P_B$ and $P_C$ in Fig.~\ref{mod8}.
\begin{figure}[H]
\centering
\includegraphics[height = 55 mm]{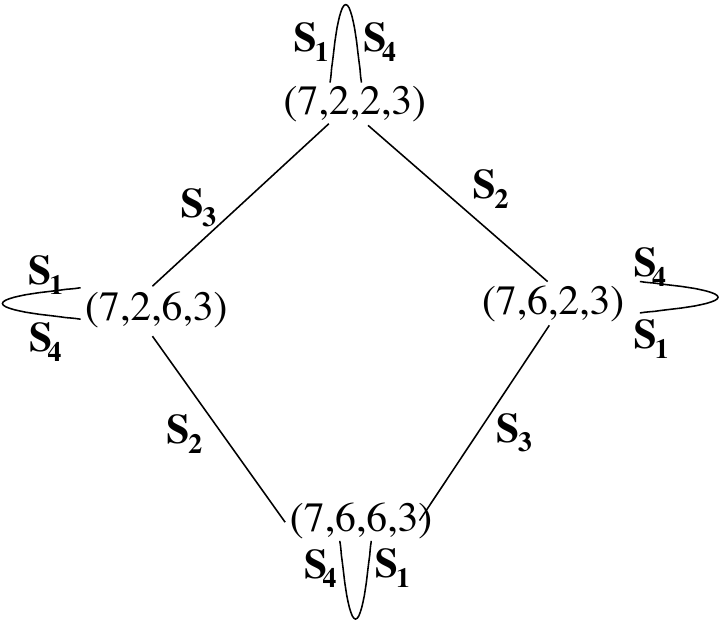}
\qquad
\includegraphics[height = 55 mm]{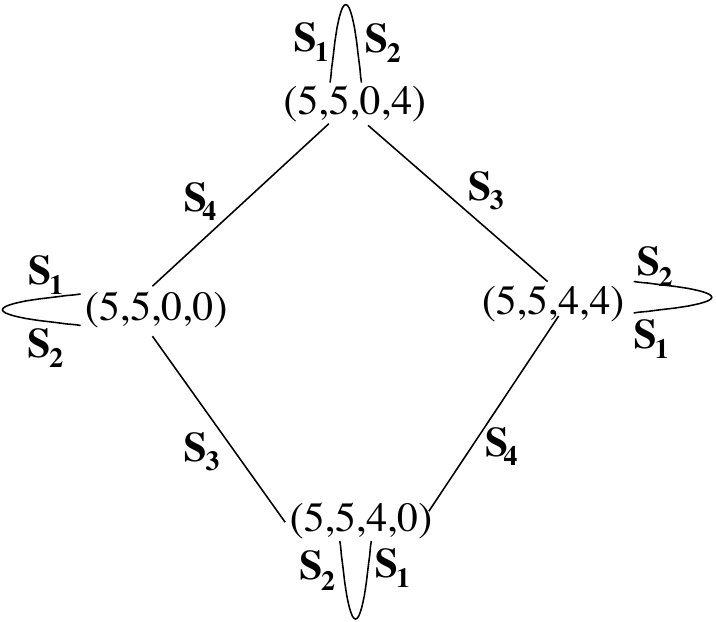}
\caption{Orbits of $P_B$ and $P_C$ modulo $8$}\label{mod8}
\end{figure}
Note that each vertex in $\mathcal G_{8}$ and $\mathcal G_3$ is connected to its neighboring vertices via all of the generators $S_i$.  Therefore the curvatures of circles in a packing modulo $24$ are equally distributed among the coordinates of the vertices in $\mathcal G_{24}$.  Combined with Theorem~\ref{padicorbit}, this lets us compute the ratio of curvatures in a packing which fall into a specific congruence class modulo $24$.  Namely, let $\mathcal{O}_{24}(P)$ be the orbit mod $24$ corresponding to a given packing $P$.  For $\mathbf w\in \mathcal{O}_{24}(P)$ let $w_i$ be the $i$th coordinate of $\mathbf w$.  We define $\gamma(n,P)$ as the proportion of coordinates in $\mathcal{O}_{24}(P)$ congruent to $n$ modulo $24$.  That is, 
\begin{equation}
\gamma(n,P) = \frac{\sum_{i=1}^4 \#\{ w \in \mathcal{O}_{24}(P) | \; w_i = n \} }{4 \cdot \#\{ \mathbf w \in \mathcal{O}_{24}(P) \} }. \label{gammadef}
\end{equation}
With this notation, a packing $P$ contains a circle of curvature congruent to $n$ modulo $24$ iff $\gamma(n,P)>0$.  Given (\ref{2483}), we express $\gamma$ as follows:
\begin{eqnarray}\label{multgamma}
\gamma(n,P)&=& \frac{\sum_{i=1}^4\#\{ \mathbf w \in \mathcal{O}_{24}(P) | \; w_i = n \} }{ 4\cdot \#\{ \mathbf w \in \mathcal{O}_{24} \} } \\\nonumber \\ &=&\frac{\sum_{i=1}^4 \#\{ \mathbf w \in \mathcal{O}_{8} | \; w_i \equiv n \; (3)\}\cdot\#\{ w \in \mathcal{O}_{3} |\; w_i \equiv n \;(8)\}}{ 4\cdot \#\{ \mathbf w \in \mathcal{O}_{8} \} \,\cdot\,  \#\{ \mathbf w \in \mathcal{O}_{3} \} }.\nonumber
\end{eqnarray}
The significance of $\gamma$ in the case of any packing (not only the two we consider) is explained in the following lemma.
\begin{lemma}\label{gammalemma}
Let $N_P(x)$ be as before, let $C$ be a circle in an integral Appolonian packing P, and let $a(C)$ be the curvature of $C$.  Then
$$\sum_{\stackrel{C\in P}{\stackrel{a(C)<x}{a(C)\equiv n\,(24)}}} 1 \sim \gamma(n,P)\cdot N_P(x)$$
\end{lemma}
This follows from Theorem~\ref{padicorbit}.  Note that in general the orbits $\mathcal O_8(P)$ and $\mathcal O_3(P)$ have $4$ and, respectively, $10$ vertices in the corresponding finite graphs\footnote[4]{There are only two possible orbits mod $3$, but many more mod $8$.  We examine just two such orbits here.}.  Therefore $\mathcal G_{24}$ always has $40$ vertices, and the ratio in (\ref{multgamma}) is easily computed using this graph.  With this in mind, we observe the following about the packing $P_B$.

\begin{lemma}\label{bugob}
Let $P_{B,24}$ denote the possible congruence classes of curvatures mod $24$ in the packing $P_B$, and let $N_{P_B}(x)$ be as in Theorem~\ref{oh}.  Then we have
\begin{itemize}
\item[(i)]$N_{P_B}(x)\sim c_{P_B}\cdot x^{\delta}$, where $c_{P_B}=0.402\dots$
\item[(ii)] $P_{B,24}=\{2,3,6,11,14,15,18,23\}$
\vspace{0.6in}
\item[(iii)] 
\vspace{-0.45in}
\begin{eqnarray*}
\gamma(2,P_B) = \frac{3}{20} &\qquad& \gamma(14,P_B) = \frac{3}{20}  \\
\gamma(3,P_B) = \frac{1}{10} &\qquad& \gamma(15,P_B) = \frac{1}{10}  \\
\gamma(6,P_B) = \frac{1}{10} &\qquad& \gamma(18,P_B) = \frac{1}{10}  \\
\gamma(11,P_B) = \frac{3}{20} &\qquad& \gamma(23,P_B) = \frac{3}{20}.
\end{eqnarray*}
\vspace{0.2in}
\item[(iv)] For $10^6<x<5\cdot 10^8$, let $x_{24}$ denote $x$ mod $24$.  If $x_{24}\in P_{B,24}$ then $x$ is a curvature in the packing $P_B$.
\end{itemize}
\end{lemma}
Part (iv) is an observation based solely on our computations using the algorithm described in Section~\ref{algorithm} -- these are illustrated in the histograms below.  The first three parts follow from computations combined with Theorem~\ref{oh} and Lemma~\ref{gammalemma}.  Note that $\gamma(n,P_B) = \gamma(n+12,P_B)$ -- for this particular packing, one can hence express the local obstructions modulo $12$ rather than modulo $24$.  Whenever this is the case for an integral ACP, we will find that there are eight congruence classes modulo $24$ in the curvatures of the circles.  Graham et.al. observe this in \cite{Apollo} where they compute which integers less than $10^6$ are ``exceptions" for $P_B$-- they find integers which satisfy these local conditions for $P_B$ modulo $12$ but do not occur as curvatures in the packing\footnote[5]{There is a small error in the computations of Graham et.al. -- we have found that the integer $13806\equiv 6\;(12)$ does not appear as a curvature in $P_B$.  Their results do not reflect this.}.  Our data extends their findings (we consider integers up to $5 \times 10^8$), and shows that all integers between $10^6$ and $10^7$  belonging to one of the congruence classes in part (ii) of Observation~\ref{bugob} appear as curvatures in the packing $P_B$.

The following histograms illustrate the distribution of the frequencies with which each integer in the given range satisfying the specified congruence condition occurs as a curvature in the packing $P_B$.  The frequencies seem normally distributed, and the means of the distributions can be computed, as we do in (\ref{meaneq}).  The variance, however, is much more difficult to predict at this time -- explaining the behavior of the variance as we consider larger integers would shed more light on our local to global conjecture.  Note that there are no exceptions to the local to global principle in this range whenever $0$ is not a frequency represented in the histogram (i.e. each integer occurs at least once).  There are several other frequencies not represented in the histograms for both $P_B$ and $P_C$ (these show up as gaps in the graphs), and an explanation of this aspect would be interesting.
\vspace{1in}

\includegraphics[height=54mm]{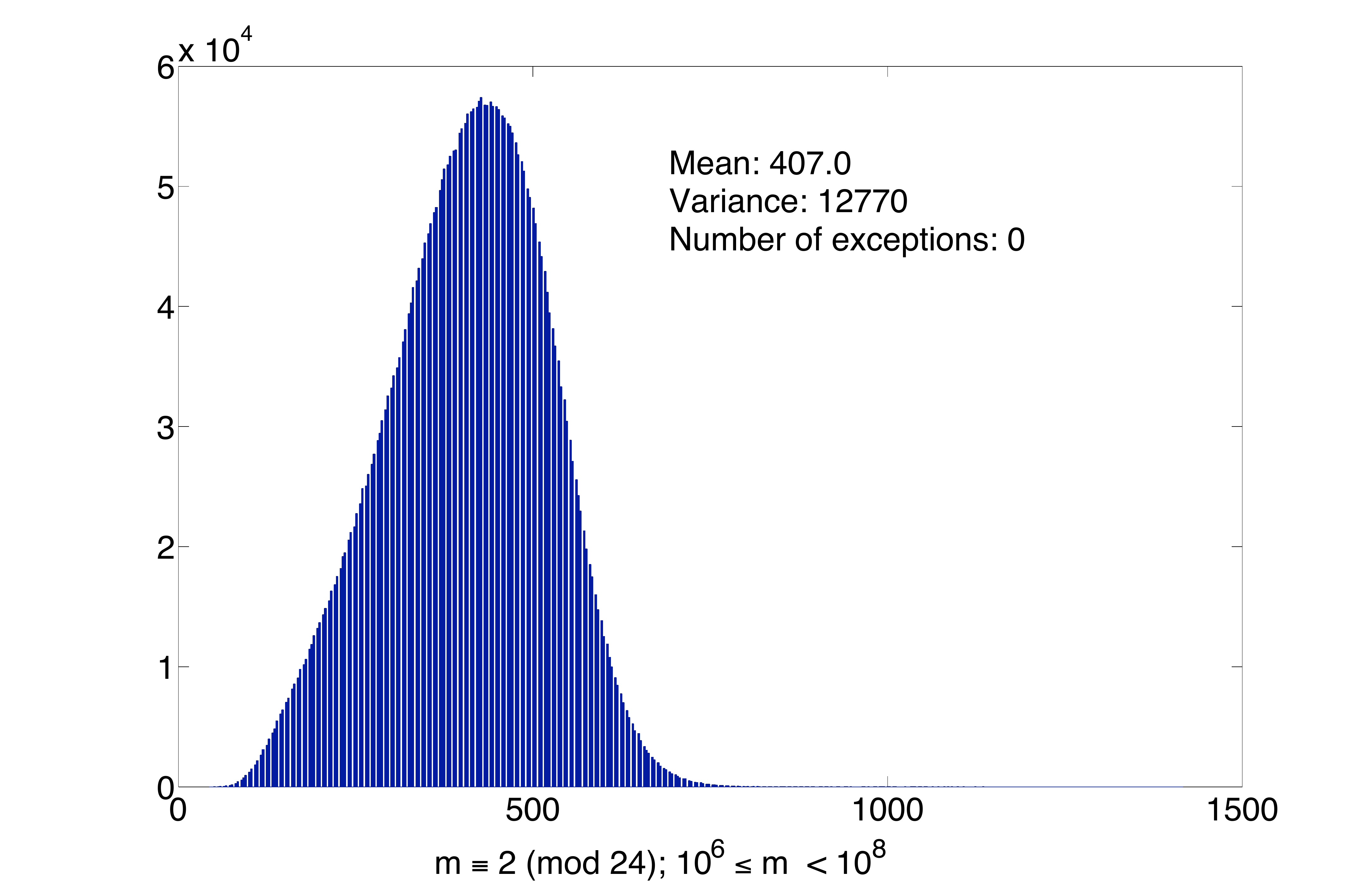}
\hfill
\includegraphics[height=54mm]{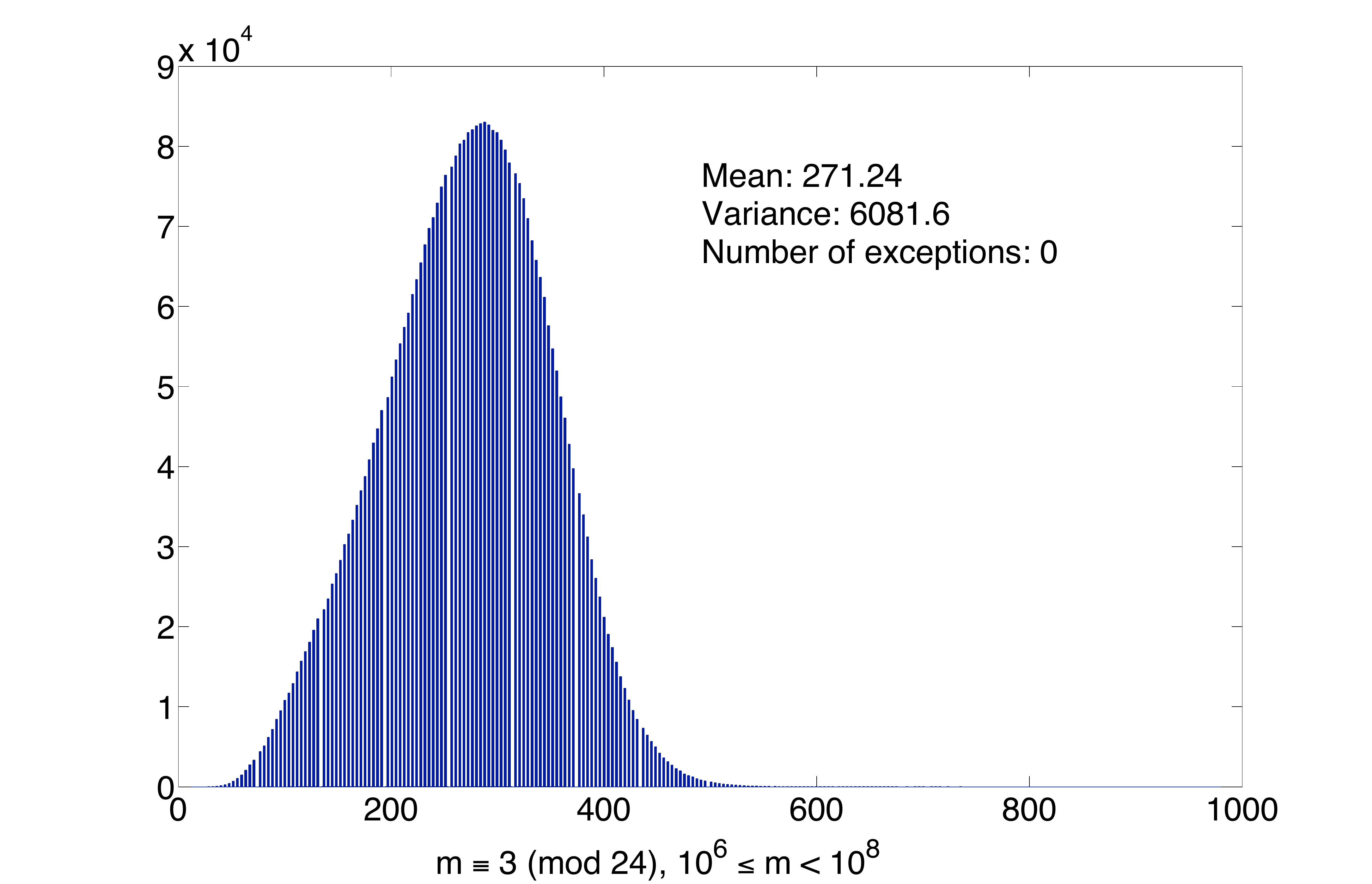}

\includegraphics[height=54mm]{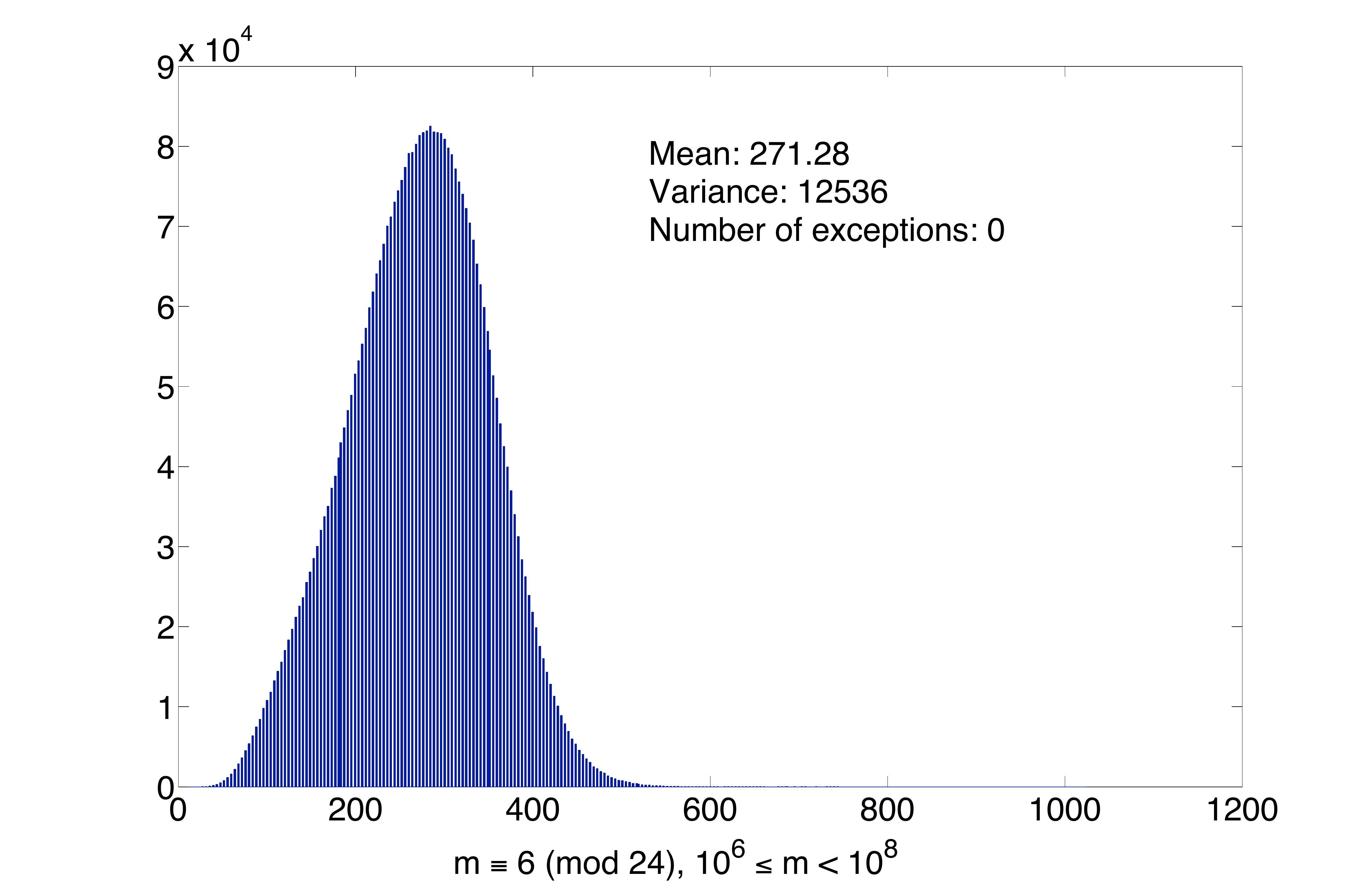}
\hfill
\includegraphics[height=54mm]{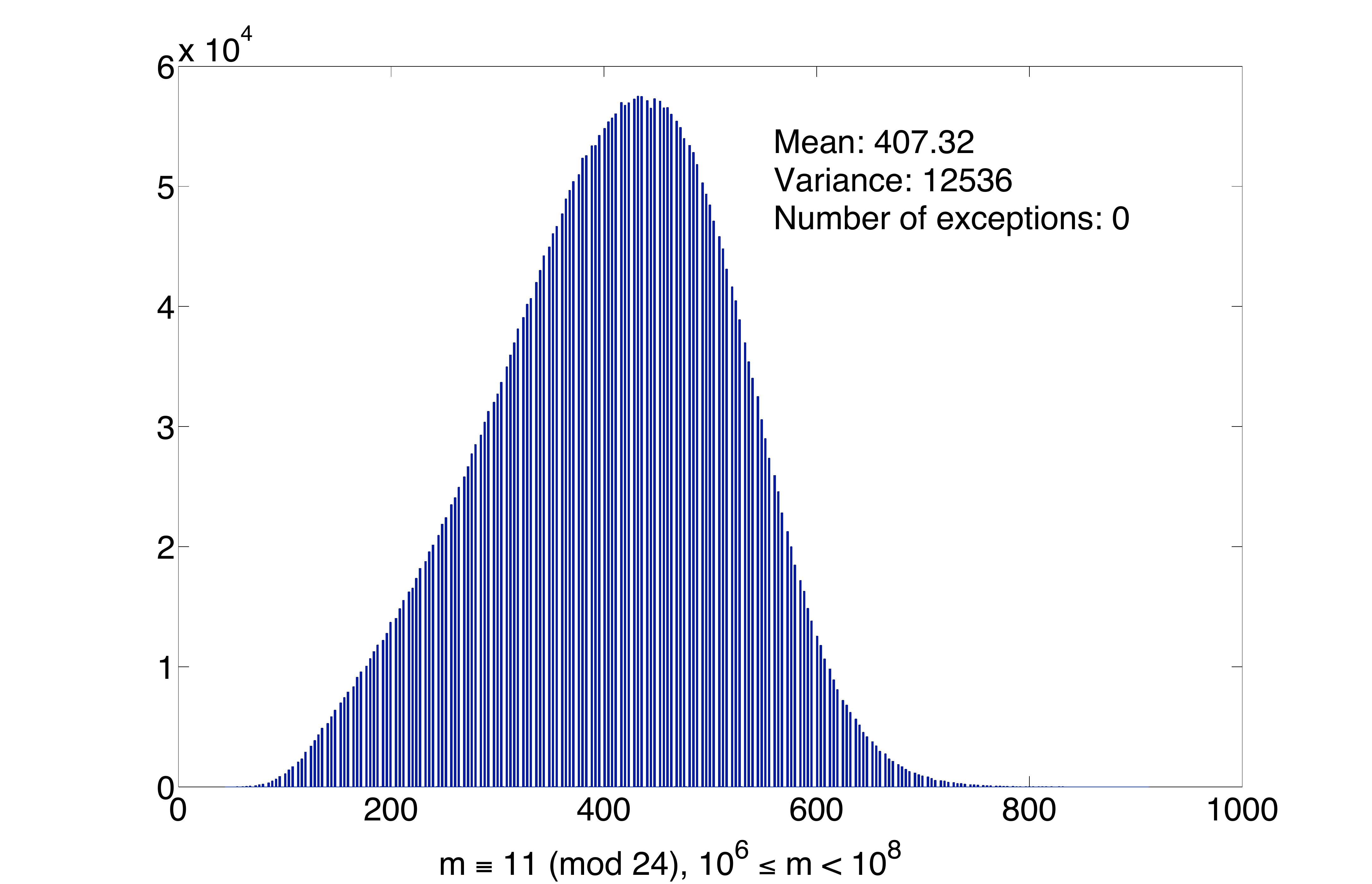}

\includegraphics[height=54mm]{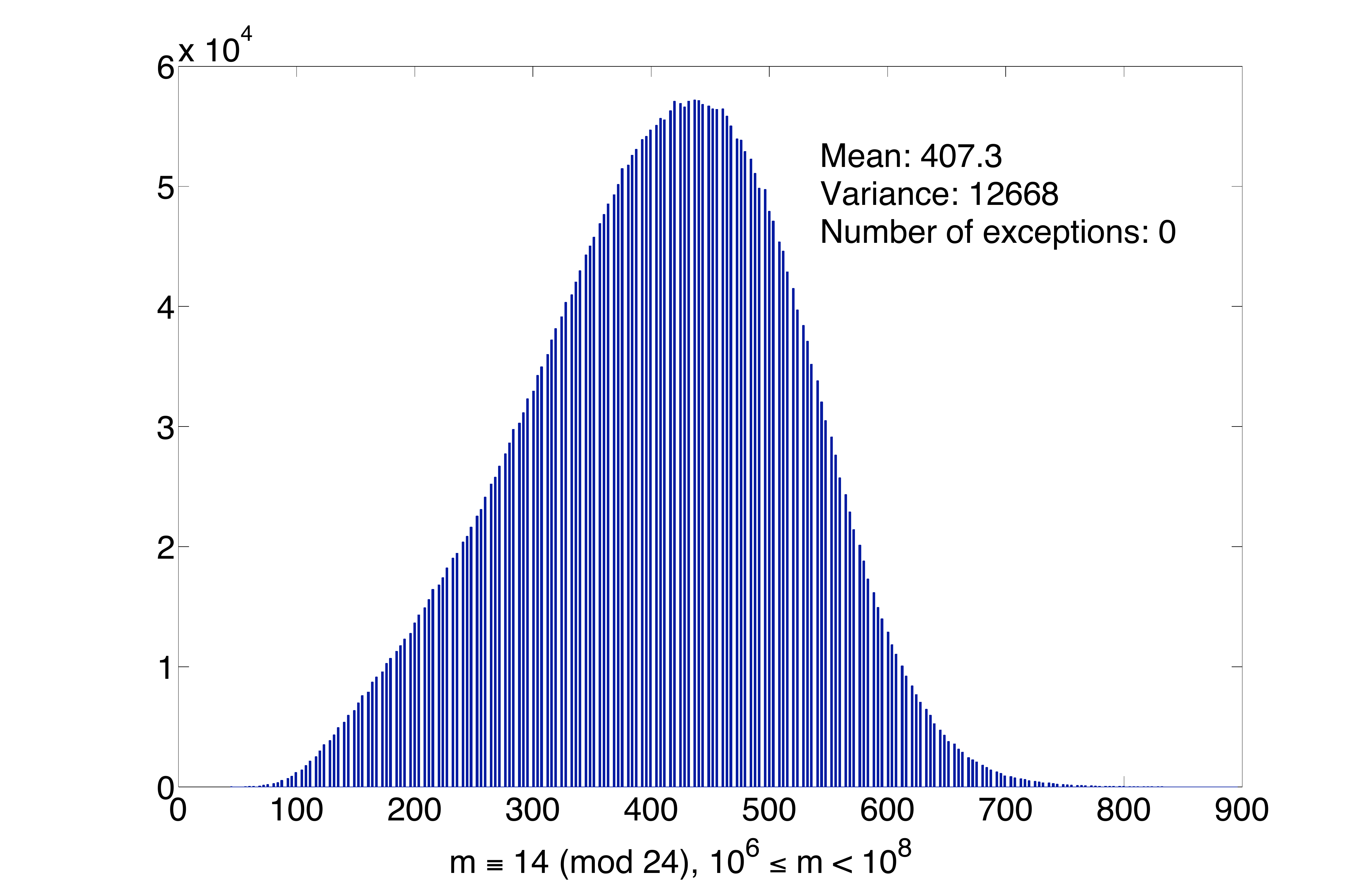}
\hfill
\includegraphics[height=54mm]{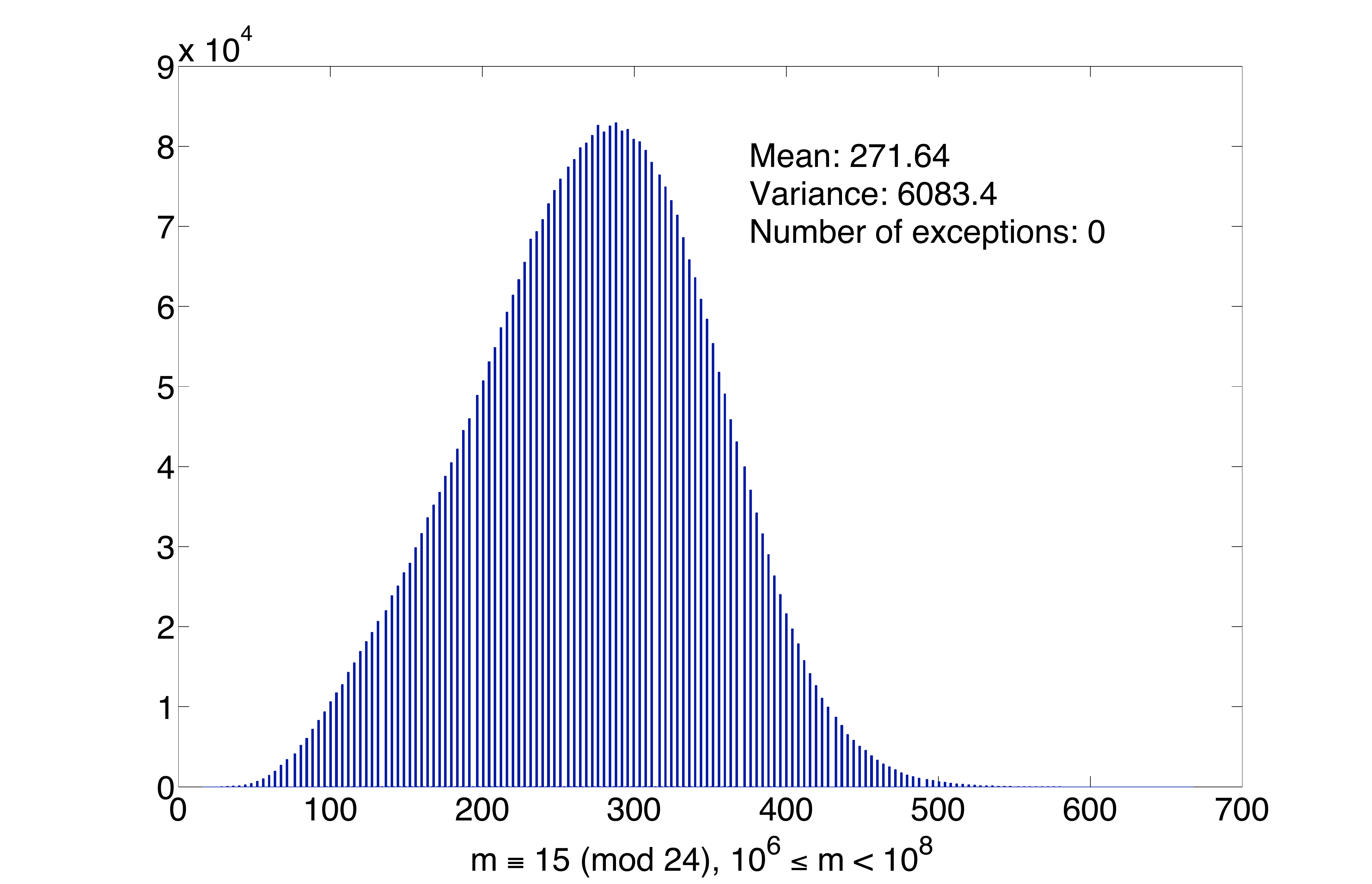}

\includegraphics[height=54mm]{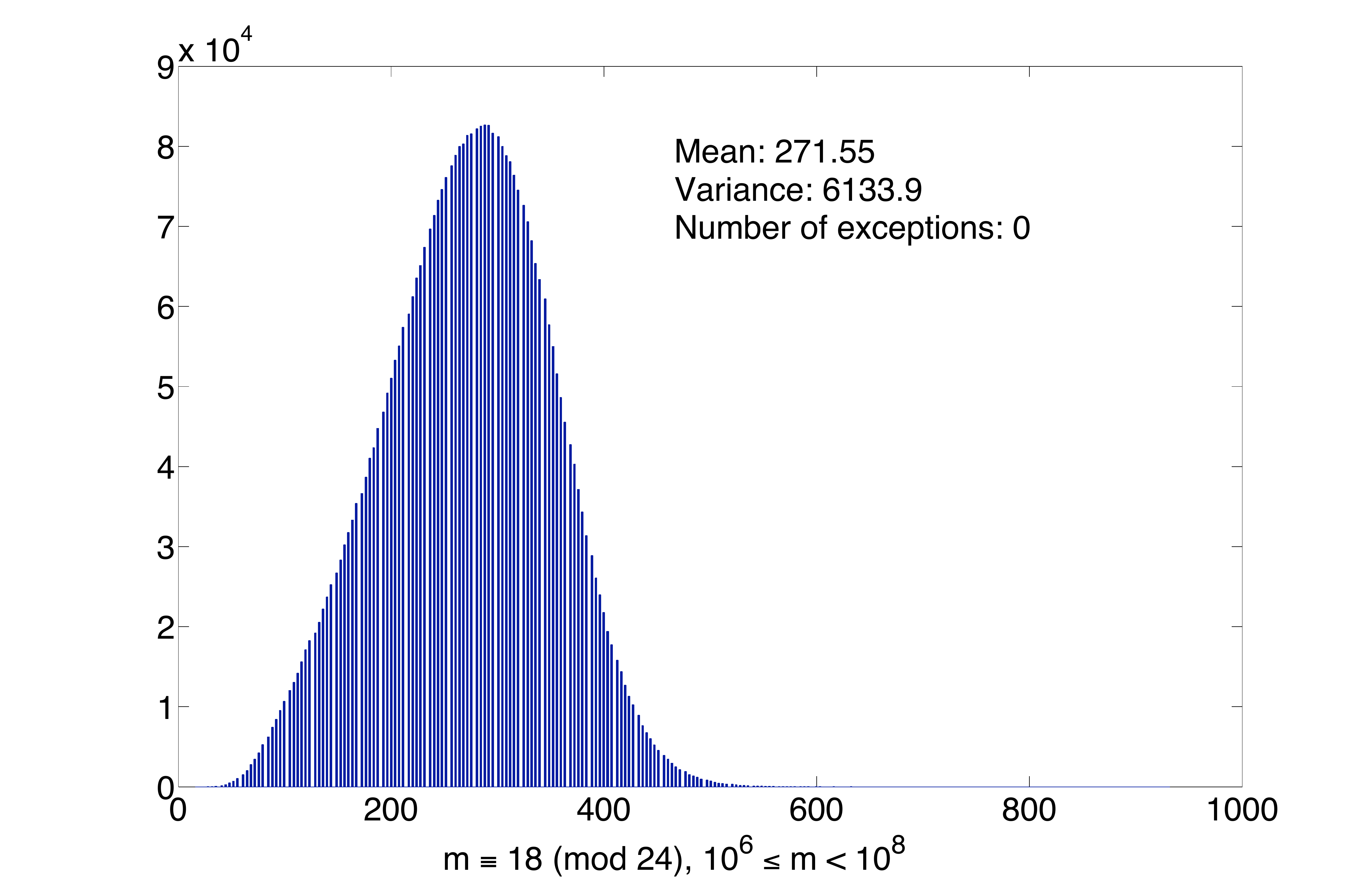}
\hfill
\includegraphics[height=54mm]{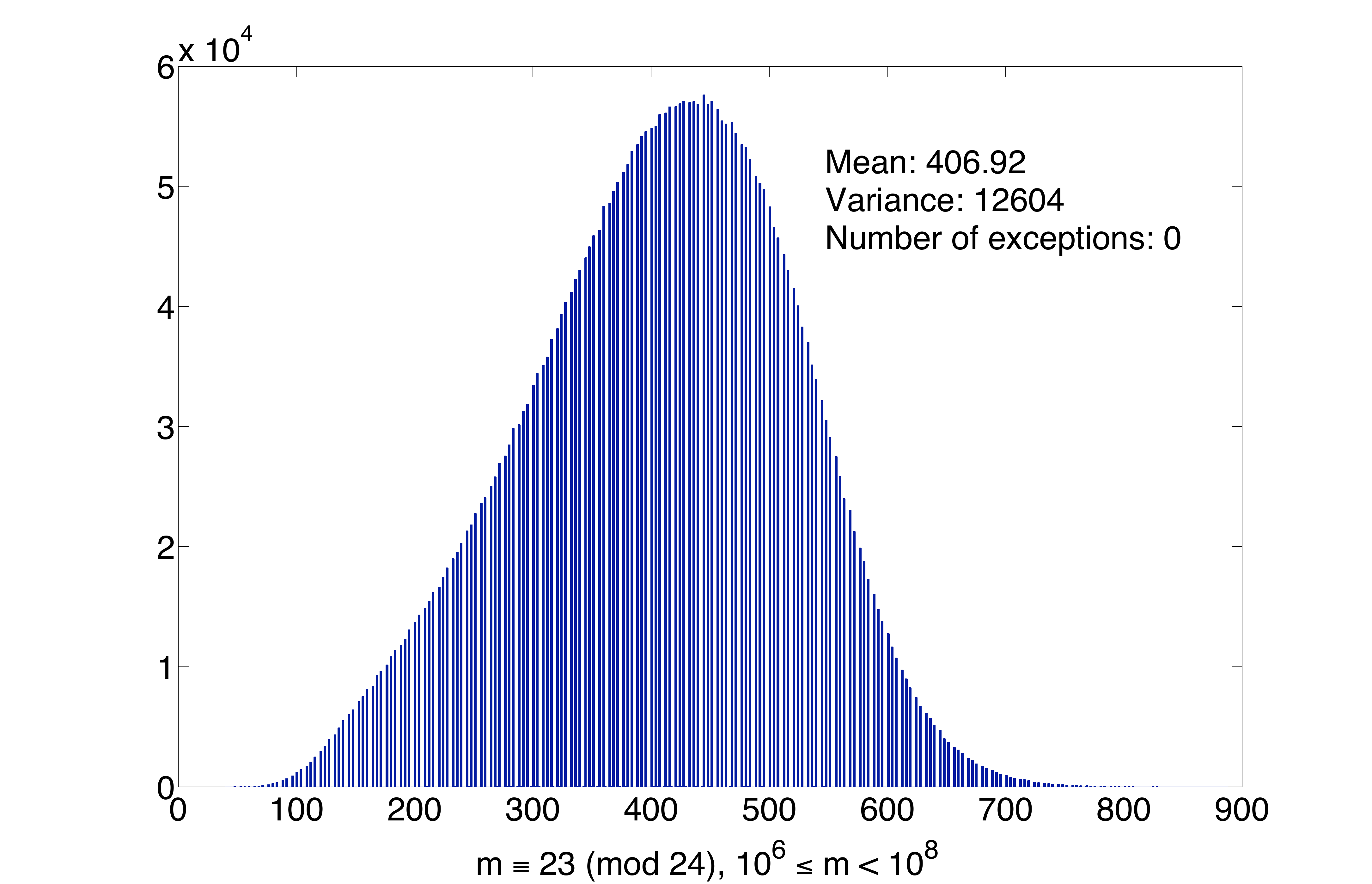}

We do the same analysis for the packing $P_C$.  In this case, we must consider much larger integers than in the case of $P_B$ in order to get comparable results.  This can be explained partially by the fact that the constant $c_P$ in Kontorovich-Oh's formula
\begin{equation}\label{c}
N_P(x)\sim c_P\cdot x^\delta
\end{equation}
is much smaller for the packing $P_C$ than the packing $P_B$ since the initial four circles in $P_C$ are much larger than the initial four in $P_B$ (See part (i) of Lemmata~\ref{bugob} and \ref{coinob}).  Specifically, $c_{P_C}=0.0176\dots$ and $c_{P_B}=0.402\dots$.  However, our data suggests that the proposed local to global principle should hold for this packing as well.
\begin{lemma}\label{coinob}
Let $P_{C,24}$ denote the possible congruence classes mod $24$ in the packing $P_C$, and let $N_{P_C}(x)$ be as in Theorem~\ref{oh}.  Then we have
\begin{itemize}
\item[(i)]$N_{P_C}(x)\sim c_{P_C} \cdot x^{\delta}$, where $c_{P_C}=0.0176\dots$
\item[(ii)] $P_{C,24}=\{ 0,4,12,13,16,21\}$
\vspace{0.4in}
\item[(iii)]
\vspace{-0.3in}
\begin{eqnarray*}
\gamma(0,P_C) = \frac{1}{10} &\qquad& \gamma(13,P_C) = \frac{3}{10}  \\
\gamma(4,P_C) = \frac{3}{20} &\qquad& \gamma(16,P_C) = \frac{3}{20}  \\
\gamma(12,P_C) = \frac{1}{10} &\qquad& \gamma(21,P_C) = \frac{4}{20} \\
\end{eqnarray*}
\item[(iv)] For $10^8<x<5 \,\cdot\,10^8$, let $x_{24}$ denote $x$ mod $24$.  If $x_{24}=13$ or $x_{24}=21$, then $x$ is a curvature in the packing $P_C$. 
\end{itemize}
\end{lemma}
Again, note that part (iv) is an observation based solely on our computations, while the first three parts in Lemma~\ref{coinob} rely on Lemma~\ref{gammalemma} and Theorem~\ref{oh}.  The histograms below illustrate the distribution of the frequencies with which each integer in the given range satisfying the specified congruence condition occurs as a curvature in the packing $P_C$.  Note that, as with $P_B$, the frequencies with which integers are represented in the packing seem to have a normal distribution.  However, since the mean of this distribution is much smaller for $P_C$ than for $P_B$, we find that $0$ is often a frequency represented in the histograms, and so there are still some exceptions to the proposed local to global principle in the range we consider.

\includegraphics[height=54mm]{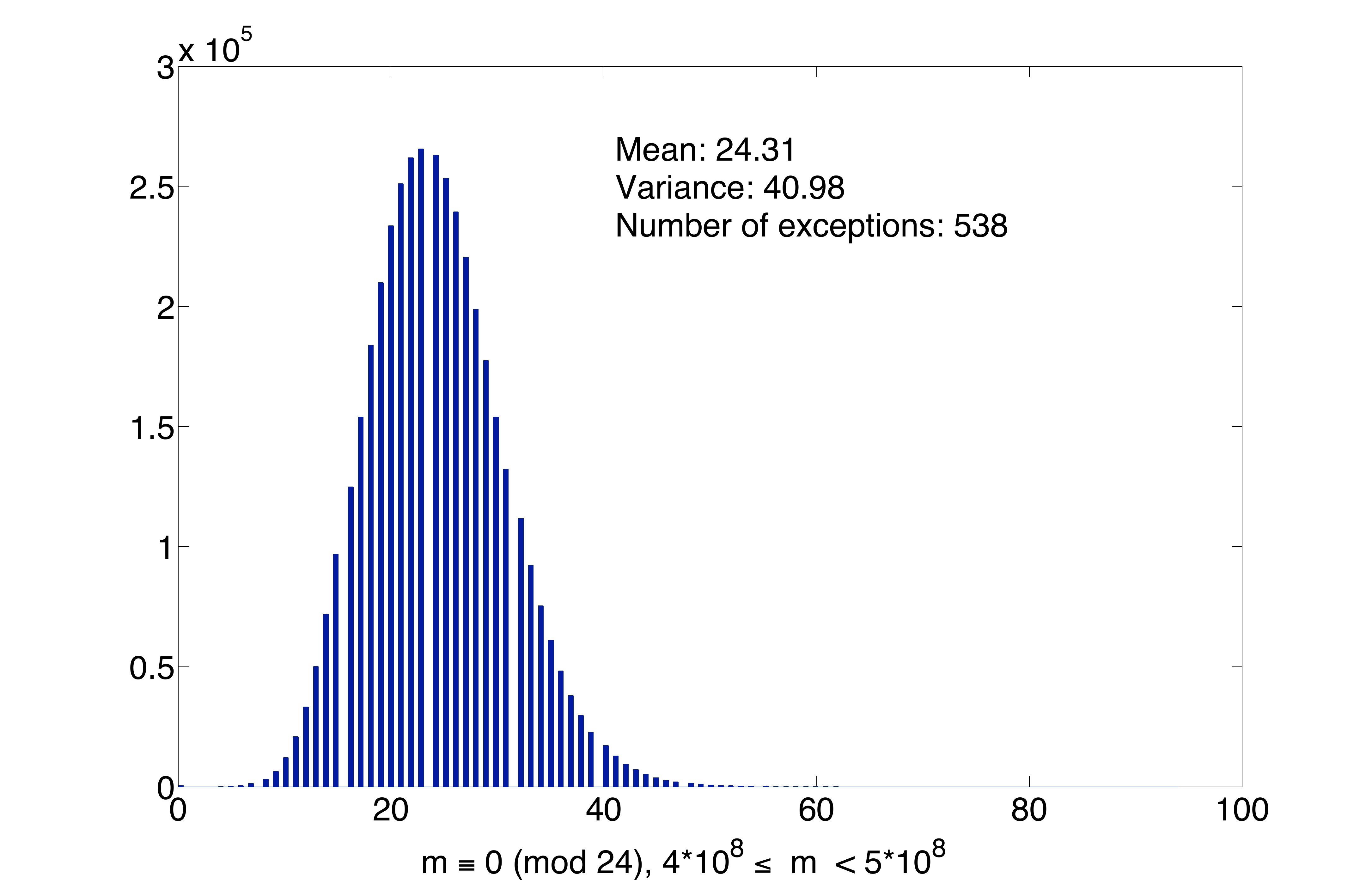}
\hfill
\includegraphics[height=54mm]{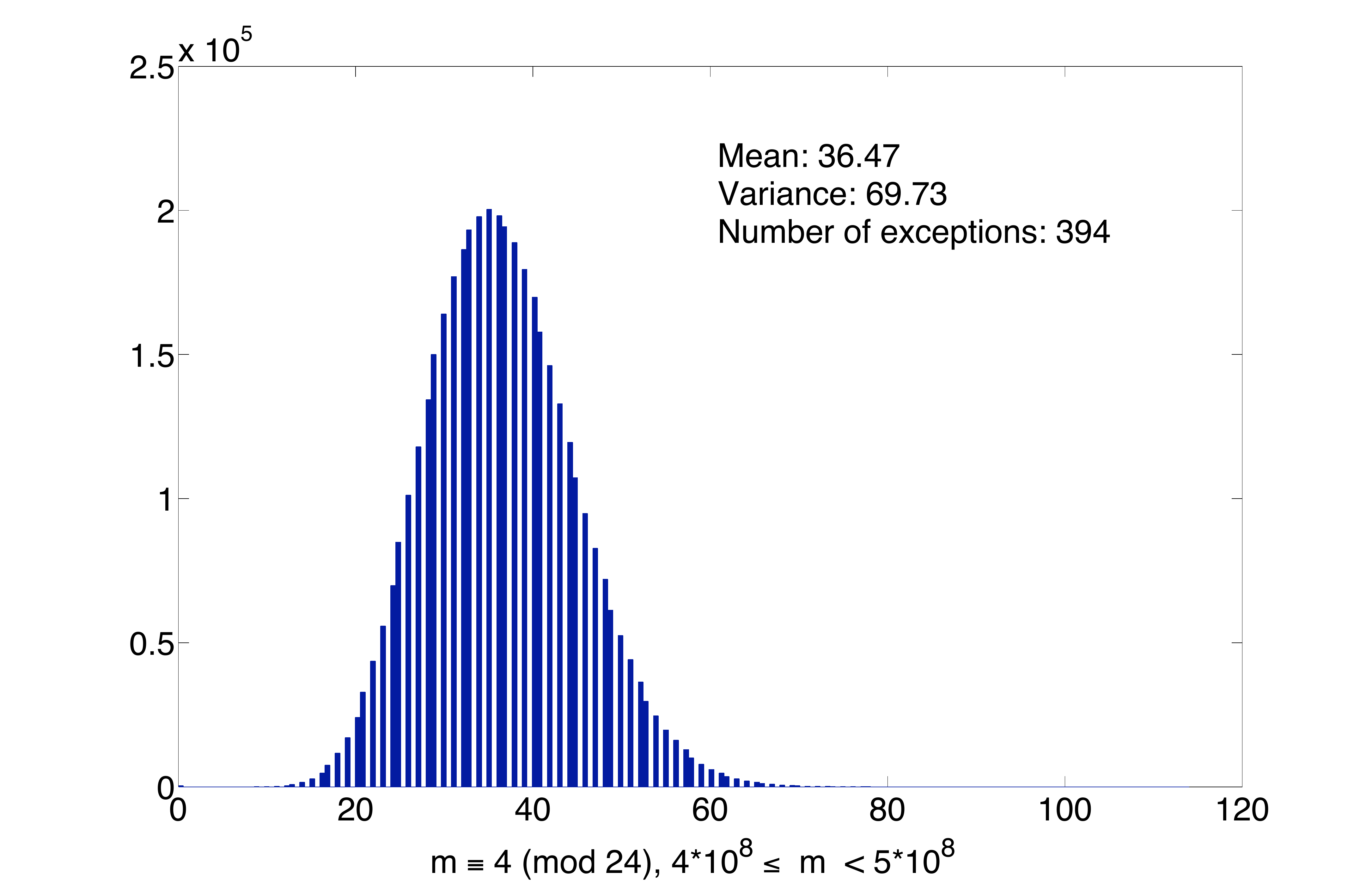}

\includegraphics[height=54mm]{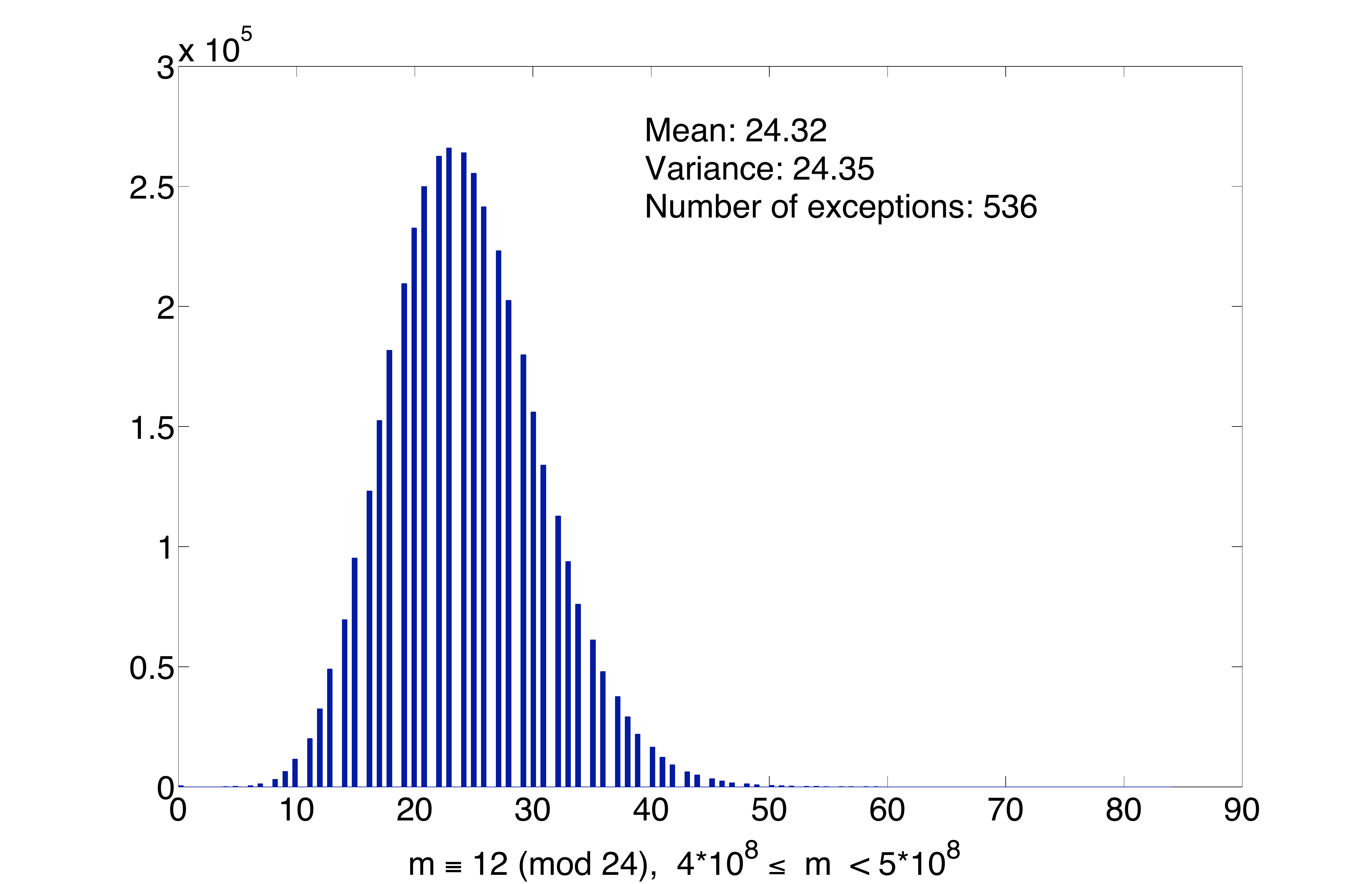}
\hfill
\includegraphics[height=54mm]{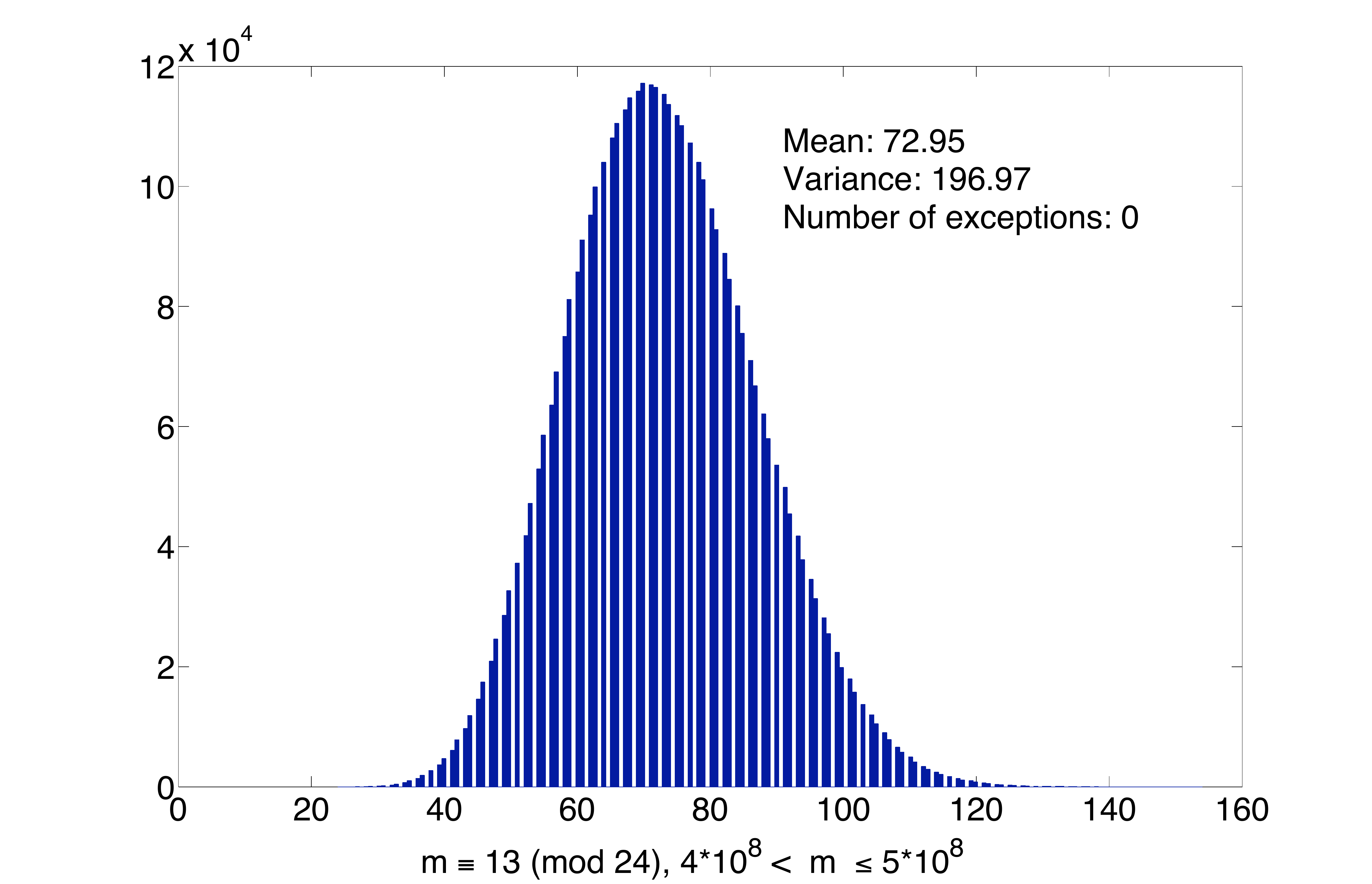}

\includegraphics[height=54mm]{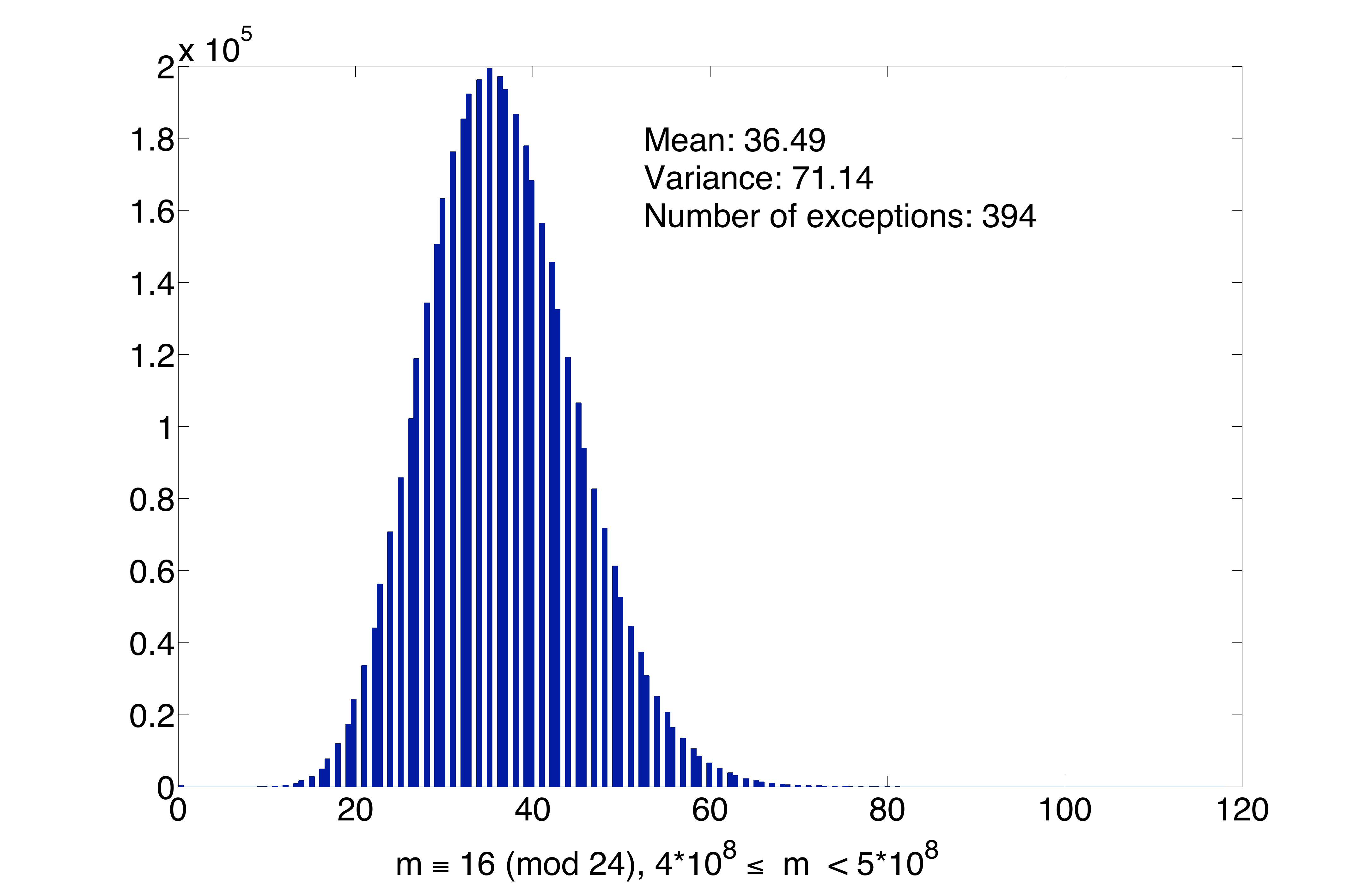}
\hfill
\includegraphics[height=54mm]{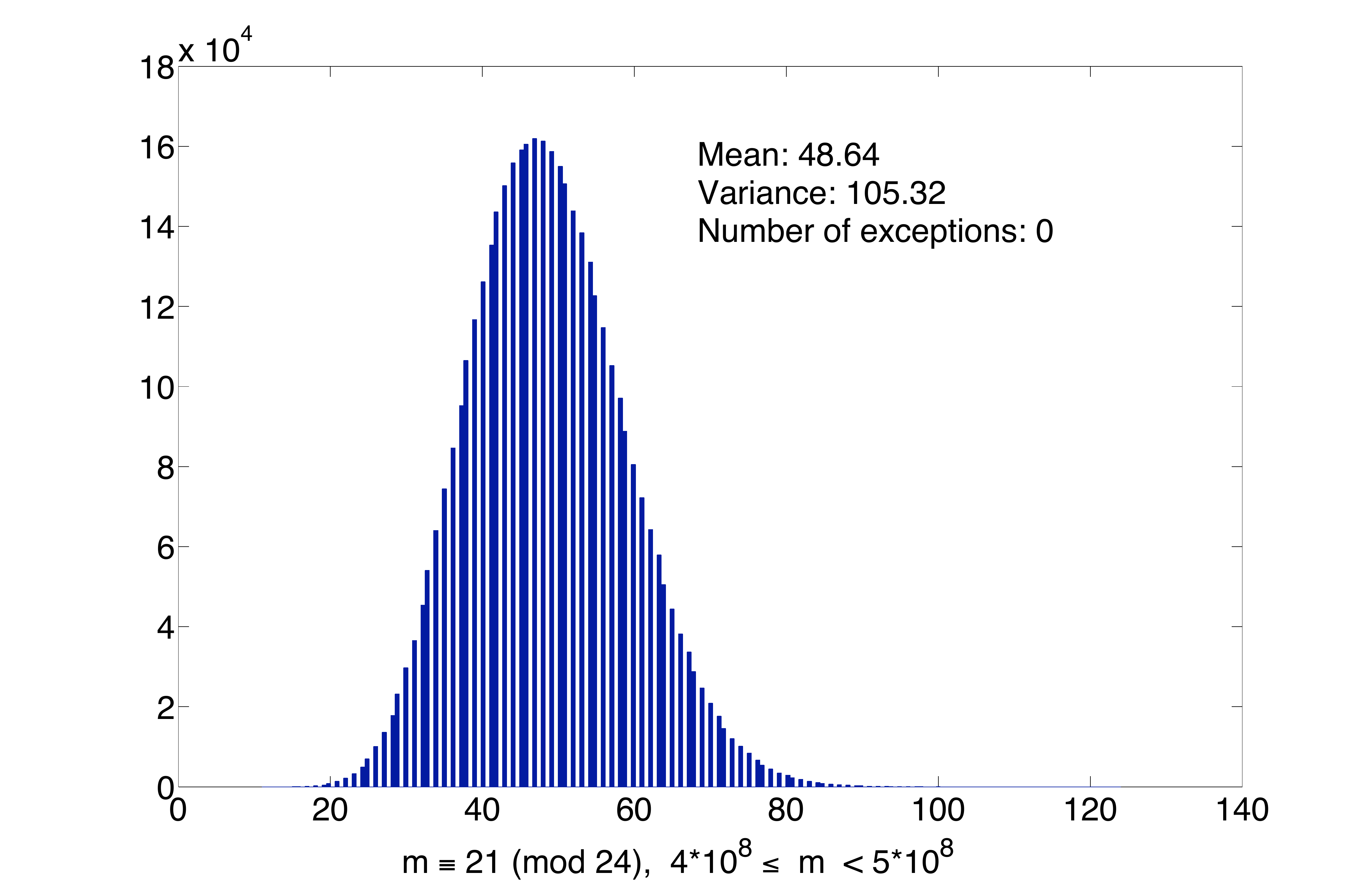}

As we mentioned before, the mean in each of these histograms is easily computable: let $C$ denote a circle in an Apollonian packing $P$ and let $a(C)$ denote the curvature of $C$.  Let $I = [k, k+K]$ be an interval of length $K$ and let $x\in I$ be an integer.  Let
$$\nu(x)=\#\{C\in P \, | \, a(C)=x\}$$
be the number of times $x$ is a curvature of a circle in $P$.  For and integer $m\geq 0$, let
$$\delta(m,n)=\#\{x\in I \, | x\equiv n \, (24),\, \nu(x) = m\}.$$
Then by Lemma~\ref{gammalemma},
\begin{equation}\label{equivalentsums}
\sum_{\stackrel{x\in I}{x\equiv n \, (24)}}\nu(x) = \sum_{m\geq 0}\delta(m,n)\cdot m.
\end{equation}
The equivalence of the two sums above is easy to observe -- one counts the same set of curvatures, but partitions them differently.  In particular, the expression in (\ref{equivalentsums}) allows us to determine the mean of the distributions in the histograms above.
Namely, denote by $x\in I$ an integer in some interval $I = [k, k+K]$ of length $K$.  Let $1\leq n\leq 24$, and let $\mu(n,P)$ denote the mean of the number of times $x\equiv n$ mod $24$ is represented as a curvature in the packing $P$.  Note that there are precisely $K/24$ integers congruent to $n$ mod $24$ in the interval $I$.  Combined with (\ref{equivalentsums}), this gives us
\begin{equation}\label{meaneq}
\mu(n,P)\approx\frac{24\cdot \gamma(n,P)\cdot(N_P(k+K)-N_P(k))}{K}.
\end{equation}
This formula predicts the following values for the means in $P_B$ in the range $[10^6, 10^8)$; and $P_C$ in the range $[4 \times 10^8, 5 \cdot 10^8)$:
\begin{eqnarray*}
&&\mu(2,P_B)=\mu(11,P_B)=\mu(14,P_B)=\mu(14,P_B)=\mu(23,P_B) = 406.70\dots\\
&&\mu(3,P_B)=\mu(6,P_B)=\mu(15,P_B)=\mu(18,P_B) = 271.13 \dots\\
&&\mu(0,P_C)=\mu(12,P_C)= 24.35\dots\\
&&\mu(4,P_C)=\mu(16,P_C)= 36.52\dots\\
&&\mu(13,P_C)=73.05\dots\\
&&\mu(21,P_C)=48.70\dots
\end{eqnarray*}
which coincides with the means observed in the histograms.  This clarifies why the mean is small for packings where the constant $c_P$ in the formula $N_P(x)\sim c_P\cdot x^{\delta}$ is small, and why one needs to consider very large integers to see that the local to global principle for such packings should hold.

This analysis can be carried out for any ACP, and will likely yield similar results.  In the direction of proving Conjecture~\ref{LG}, one might investigate how $X_P$ depends on the given packing -- can it perhaps be expressed in terms of the constant $c_P$ in (\ref{c})?  One might also ask how the variance of the distributions above depend on the packing, and how it changes with the size of the integers we consider -- answering this would give further insight into the local to global correspondence for curvatures in integer ACP's.

\section{A description of our algorithm and its running time}\label{algorithm}

We represent an ACP by a tree of quadruples.  Fig.~\ref{tree} shows the first two generations of the tree corresponding to $P_C$.  To generate all curvatures of magnitude less than $x$, we use a LIFO (last-in-first-out) stack to generate and prune this tree.  The algorithm is as follows:
\begin{itemize}
\item[\it{1)}] Push the root quadruple onto the stack.
\item[\it{2)}]Until the stack is empty, perform an iterative process:
\begin{itemize}
\item[\it{a)}] Pop a quadruple off of the stack and generate its children. 
\item[\it{b)}] For each child, if the new curvature created (i.e. the maximum entry of the quadruple) is less than $x$, then push the child onto the stack.
\end{itemize}
\end{itemize}
\vspace{.5 cm}

By pushing a quadruple onto the stack only if its maximum entry is less than $x$, we effectively prune the the tree.  Since we know each quadruple has a larger maximum entry than its parent, we use step 2\emph{b)} to avoid generating branches whose quadruples are known to have entries greater than $x$.

Although we use the concept of a tree to \emph{generate} curvatures, we note that the entire tree structure is not necessary to \emph{store} such curvatures.   Instead, we store the curvatures in a one-dimensional array of $x$ elements, all initialized to zero.  The $i$th element of the array contains the number of curvatures with magnitude $i$.  For instance, the $24$th element of the array for $P_C$ is equal to $1$, while the $25$th element is equal to $0$, since there are no curvatures equal to 25 in $P_C$.

We use these arrays to generate the histograms in Section~\ref{locglobal}.  Due to Matlab's memory constraints, we limit our Matlab arrays to $10^8$ entries.  So, to check for exceptions in the entire range of $[10^6, 5\cdot 10^8)$, we check each of the intervals $[10^6, 10^8)$, $[10^8, 2\cdot 10^8), \dots, [4\cdot 10^8, 5\cdot 10^8)$ individually.  We have chosen to display the interval $[10^6, 10^8)$ in our figures in Section ~\ref{locglobal}.

To count primes less than $x$, we simply increment a sum whenever a prime curvature is produced.  To count kissing primes less than $x$, we increment a sum whenever a prime curvature is produced \emph{and} some other member of the curvature's quadruple is prime.

It takes our algorithm $\textrm{O}(N_P(x))$ steps to compute $N_P(x)$, which is optimal since each node on the tree must be visited; that is, it is not possible to skip any quadruples.   

Our programs rely on Wayne and Sedgewicks's Stack data type and standard draw library \cite{sw}.

\begin{figure}[H]
\includegraphics[width=80mm]{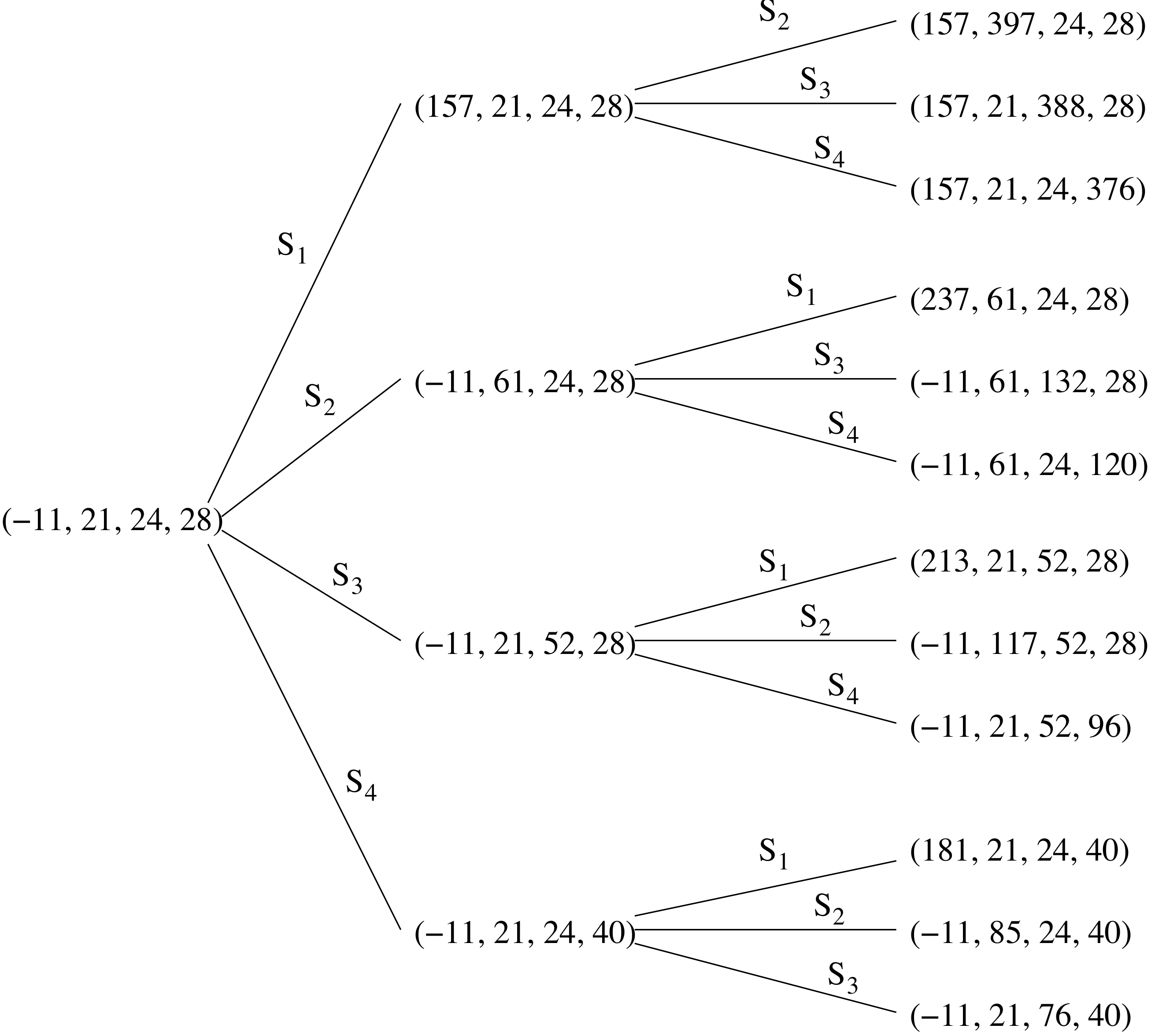}
\caption{The tree of quadruples for $P_C$, pictured up to 2 generations.}\label{tree}
\end{figure}


\end{document}